\documentclass[11pt,final]{amsart}

\usepackage{geometry,enumitem}
\usepackage{hyperref}

\usepackage[usenames, dvipsnames]{color}
\usepackage{accents}

\usepackage{amsthm, amsxtra,amssymb}

\newtheorem{theorem}{Theorem}[section]
\newtheorem{proposition}[theorem]{Proposition}
\newtheorem{lemma}[theorem]{Lemma}
\newtheorem{corollary}[theorem]{Corollary}
\newtheorem{definition}{Definition}[section]
\newtheorem{remark}[theorem]{Remark}

\numberwithin{equation}{section}

\numberwithin{equation}{section}

\newcommand{\R}{\mathbb{R}}
\newcommand{\nn}{\nonumber}

\author{Samer Dweik, Nassif Ghoussoub, Young-Heon Kim and Aaron Zeff Palmer}

\address{* Samer Dweik, Nassif Ghoussoub, Young-Heon Kim and Aaron Zeff Palmer}
\address{Department of Mathematics\\ University of British Columbia\\ Vancouver, V6T 1Z2 Canada}
\email{dweik@math.ubc.ca, nassif@math.ubc.ca, yhkim@math.ubc.ca, azp@math.ubc.ca}

\title{Stochastic optimal transport with free end time}
\date{\today}

\thanks{N. Ghoussoub and Y.-H. Kim are partially supported by  the
Natural Sciences and Engineering Research Council of Canada (NSERC). \\
\copyright 2019 by the author.
}

\begin{document}

\begin{abstract}
We consider a stochastic transportation problem between two prescribed probability distributions (a source and a target) over processes with general drift dependence and with free end times. First, and in order to establish a dual principle, we associate two equivalent formulations of the primal problem  in order to guarantee its convexity and lower semi-continuity with respect to the source and target distributions. 
We exhibit an equivalent Eulerian formulation, whose dual variational principle is given by Hamilton-Jacobi-Bellman type variational inequalities. In the case where 
the drift is bounded,
regularity results on the minimizers of the Eulerian problem then enable us to prove attainment in the corresponding dual problem. We also address attainment
when the drift component of the cost defining  Lagrangian $L$ is superlinear $L \approx |u|^p$ with $1<p<2$, in which case the setting is reminiscent of our approach -in a previous work- on  deterministic controlled transport problems with free end time. We finally address criteria under which the optimal drift and stopping time are unique, namely strict convexity in the drift component and monotonicity in time of the Lagrangian.
\end{abstract}
\maketitle

\tableofcontents

\section{Introduction}

The problem of transporting a distribution from a given source to a prescribed target has been studied since the pioneering work of Monge in 1781 (see \cite{Monge}) and has many applications in analysis,
probability theory, and partial differential equations. In this paper, we consider an optimal transportation problem for stochastic processes with controlled dynamics and free end time, where the transport cost is given by a general Lagrangian $L$ on $\R^+\times \R^d\times \R^d$ as follows:
if $\mu$ and $\nu$ are two probability measures, then the stochastic transport problem is formally stated as
\begin{align}\label{primal}
{\mathcal P}_L(\mu, \nu):=    \inf_{\beta,\tau} \Big\{\mathbb{E}\Big[\int_0^\tau L\big(t,X_t,\beta_t\big)dt\Big];\ dX_t=\beta_t\, dt+dW_t,\ X_0\sim \mu,\ X_\tau\sim \nu\Big\}.
\end{align}
More precise definitions will be given later, but for now we mention that the minimization is over all suitable drifts $\beta$, and all  stopping times $\tau$.  The notation $X_0 \sim \mu$ means that the initial position of the path has $\mu$ as its distribution, and $X_\tau \sim \nu$ means that the processed stopped at the random time $\tau$ has $\nu$ as its distribution.

In \cite{mikami2002optimal,mikami2008optimal,tan2013optimal}, the authors consider Problem \eqref{primal} in the case where the end time is fixed ($\tau=1$), i.e. they minimize
 \begin{align}\label{Primal with fixed end time}
\inf_{\beta} \Big\{\mathbb{E}\Big[\int_0^1 L\big(t,X_t,\beta_t\big)dt\Big];\ dX_t=\beta_t\, dt+dW_t,\ X_0\sim \mu,\ X_1\sim \nu\Big\}.
\end{align}
 Under some assumptions on the Lagrangian $L$, they establish a weak duality  principle for  \eqref{Primal with fixed end time}, namely that the primal value of \eqref{Primal with fixed end time} equals the value of the following dual problem:
\begin{align} \label{dual with fixed end time} 
\sup_{\psi}\Big\{\int_{\mathbb{R}^d} \psi(y)\,\nu(\mathrm{d}y) - \int_{\mathbb{R}^d} J(0,x)\,\mu(\mathrm{d}x);\ J \mbox{\ solves\ }\eqref{PDE with fixed end time} \Big\}
\end{align}
where 
\begin{equation} \label{PDE with fixed end time}
\begin{cases}
\partial_t J(t,x)  + \frac{1}{2} \Delta J(t,x) + H\big(t,x,\nabla J(t,x)\big)= 0 &\  \mbox{in\ }\, (0,1) \times \mathbb{R}^d,\\
J(1,x)= \psi(x) &\  \mbox{on\ }\,\mathbb{R}^d,
\end{cases}
\end{equation}
where $H$ is the Hamiltonian associated to $L$.
A special case that has received recent attention is when the Lagrangian is of the form $L(t,x,u)=\frac{1}{2}|u|^2+V(x)$, which has connections with the Schr\"{o}dinger bridge problem
\cite{leonard2012schrodinger, chen2016relation,mikami2004monge}.  The existence of optimal $\psi$ has also been established for the case when $L$ has quadratic growth in $u$ and the target distribution is smooth with $\nu>0$ for a version of the problem including a mean field cost posed on the torus   \cite{porretta2013planning,porretta2014planning}, which makes use of the variational structure and energy estimates.  On the other hand, stopping uncontrolled processes with distribution constraints has a vast literature, in particular pertaining to applications in finance; for some of the approaches related to (\ref{primal}) see  \cite{bayraktar2019distribution, beiglbock2018geometry,beiglboeck2017optimal, kallblad2017dynamic}.

Coming back to Problem \eqref{primal}, we shall pose the initial problem in a weaker sense so that it involves randomized stopping times and weak solutions to the SDE, analogous to the Kantorovich relaxation of the optimal transport \cite{kantorovich1942translocation} problem. The ultimate goal is  to obtain minimizers involving true stopping times and representing strong solutions to the SDE.
For that, we shall give two formulations of the stochastic transport problem (see Section~\ref{sec:stoch-form}), which we will ultimately prove equivalent:
\begin{itemize}
    \item A \emph{weak stochastic} formulation that poses the optimization problem over (weakly) controlled processes and randomized stopping times.
    \item A \emph{convex stochastic} formulation which poses the optimization over probability measures on a space of randomly stopped paths for both state and drift. 
\end{itemize}

Under appropriate conditions on $L$, the latter equivalent formulation renders the problem l.s.c. and convex in $\nu$.  This will allow us to identify a corresponding dual problem, which can be described as follows:
\begin{align}\label{DualStochastic}
\mathcal{D}_L(\mu,\nu)=\ \sup_{\psi}\Big\{\int_{\R^d} \psi(y)\,\nu(dy) - \int_{\R^d} J_\psi(0,x)\, \mu(dx)\Big\},
\end{align}
where  $J_\psi$ can be viewed as the viscosity solution (equivalently minimal supersolution) of  the following second order Hamilton-Jacobi-Bellman quasivariational inequality,
\begin{equation}\label{HJBESec.0}
\begin{cases}
\partial_t J(t,x) +  \frac{1}{2} \Delta J(t,x) + H\big(t,x,\nabla J(t,x)\big) \leq 0 &\ \mbox{in}\,\,\ \mathbb{R}^+ \times \mathbb{R}^d,\\
\psi(x) - J(t,x)\leq 0   &\ \mbox{on}\,\,\ \mathbb{R}^+ \times\mathbb{R}^d,
\end{cases}
\end{equation}
where $H$ is again the Hamiltonian associated to $L$. We then establish the (weak) duality principle (see Section~\ref{sec:duality}),
\begin{equation}\label{weakD}
\mathcal{D}_L(\mu,\nu)={\mathcal P}_L(\mu, \nu). 
\end{equation}
 
 The most crucial part of the analysis is to find an optimal end potential $\psi$ (hence $J_\psi$) for the dual problem \eqref{DualStochastic}. In fact, there is no general result in the literature about the attainment in the dual problem \eqref{dual with fixed end time}. However, one of the main advantages of considering \eqref{primal} instead of \eqref{Primal with fixed end time} is that the constraints on the potentials $(\psi,J_\psi)$ are now somehow relaxed, i.e. we have $J \geq \psi$ on $\mathbb{R}^+ \times \mathbb{R}^d$ instead of $J(1,\cdot)=\psi$ and so, this will allow us to replace $\psi$ by an end potential $\bar\psi\geq \psi$ that satisfies
 $$
    \frac{1}{2}\Delta \bar\psi(x)+\inf_{t\in \R^+}H\big(t,x,\nabla \bar\psi(x)\big)\leq 0,
 $$
 and we get some Sobolev/H\"older estimates on $\bar\psi$.
We note that this attainment, in tandem with the weak duality (\ref{weakD}), allows - through  a verification type theorem -  to characterize both the optimal process and the stopping time that resolve the primal problem. In our quest to prove attainment in the dual problem,
we will focus on two cases:\\

\begin{enumerate}
\item  When the drift is bounded,
the distribution of the process then possesses additional Sobolev regularity due to the diffusion. However, the dual potential may be unbounded with singularities similar to the fundamental solution of the Laplace equation. We find that a sufficient condition to solve the problem is for the target distribution to lie in the dual to an appropriately weighted $L^1$ space. This case builds upon the Sobolev space approach for the Skorokhod problem (without drift) studied in \cite{GKPS}.\\

\item When the drift is strong (that is if \,$1<p<2$), there is no additional regularity on the density due to the possibility of `local controllability', i.e.\ the ability  to transport to a Dirac-mass with finite-cost.  However, this `local controllability' allows for uniform bounds and H\"{o}lder estimates on the end potential, as in \cite{Cannarsa2010Holder}, and generalizes the approach for deterministic control problems \cite{GKP}.  In particular, any compactly supported target measure may be reached with finite cost.\\
\end{enumerate}


\noindent \hspace{0.1cm} In either case, we shall prove attainment in the dual problem (Section \ref{sec:attainment-large-p}). But for that, we need to introduce
two Eulerian formulations (see Section~\ref{sec:Eulerian}) for the weak stochastic formulation of the primal problem (\ref{primal}). 
\begin{itemize}
\item  {\em The strong Eulerian formulation}, which poses the problem with a velocity field and the solution to a Fokker-Planck equation with stopping. More precisely,
\begin{align}
    \mathcal{P}_L^{\mathcal{E}}(\mu,\nu)=\inf_{(m,v,\rho)\in \mathcal{E}(\mu)}\Big\{\int_{\R^+}\int_{\R^d} L\big(t,x,v(t,x)\big)m(t,x)\,dx dt;\ \int_{\R^+} \rho(d\tau,\cdot)=\nu\Big\},
\end{align}
where a triplet $(m,v,\rho)$ belongs to $\mathcal{E}(\mu)$ if the following hold: 
 $\rho$ is a probability measure on $\R^+\times \R^d$, $m_t$ is a nonnegative density in the Sobolev space $H^1(\R^d)$ for each time $t \in \mathbb{R}^+$, and $v$ is a measurable velocity field, and for all smooth test functions $\phi$  on $\mathbb{R}^+ \times \mathbb{R}^d$, we have the following:
    \begin{align}
    \int_{\R^+}\int_{\R^d}&\bigg[\partial_t \phi(t,x)\,m(t,x)+\nabla \phi(t,x)\cdot\bigg(v(t,x)\,m(t,x)-\frac{1}{2}\nabla m(t,x)\bigg)\bigg]dxdt\nn\\
        =&\ \int_{\R^+}\int_{\R^d} \phi(\tau,y)\rho(d\tau,dy)-\int_{\mathbb{R}^d}\phi(0,x)\mu(dx).
    \end{align}

\item  {\em The convex Eulerian formulation}, which poses the problem over phase-space distributions satisfying a convex set of inequalities. More precisely,
\begin{align}
    \mathcal{P}_L^{\tilde{\mathcal{E}}}(\mu,\nu)=\inf_{(\eta,\rho)\in \tilde{\mathcal{E}}(\mu)}\Big\{\int_{\R^+}\int_{\R^d}\int_{\R^d} L(t,x,u)\eta_t(dx,du)dt\,:\,\ \int_{\R^+} \rho(d\tau,\cdot)=\nu\Big\},
\end{align}
where a pair of density process and stopping time $(\eta,\rho) \in \tilde{\mathcal{E}}(\mu)$ if 
$\eta$ is a measurable map from $\R^+$ to nonnegative measures on $\R^d\times \R^d$,
$\rho$ is a probability measure on $\R^+\times \R^d$, and
     for all smooth test functions $\phi$ on $\mathbb{R}^+ \times \mathbb{R}^d$:
    \begin{align}
         \int_{\R^+}\int_{\R^d}\int_{\R^d}&\Big[\partial_t \phi(t,x)+\frac{1}{2}\Delta \phi(t,x)+\nabla \phi(t,x)\cdot u\Big]\eta_t(dx,du)dt\nn\\
        &=    \ \int_{\R^+}\int_{\R^d} \phi(\tau,y)\rho(d\tau,dy)-\int_{\mathbb{R}^d}\phi(0,x)\mu(dx).
    \end{align}

\end{itemize}

We shall prove  the latter to be equivalent to the primal problem (\ref{primal}) first by embedding phase-space distributions into the stochastic formulation and then showing the reverse inequality by using weak duality.  We then prove that, when the drift is bounded,
the strong Eulerian formulation is also equivalent by using suitable Sobolev estimates.

After proving attainment in the dual problem (see Section \ref{sec:attainment-large-p}), we proceed to obtain useful information on the primal problem, and in some special cases (e.g., when \,$t\mapsto L(t,x,u)$ is monotone), we show (see Section~\ref{sec:hitting_times}) that the unique optimizer is given by the hitting time to a space-time barrier,
$$
    \tau^*=\inf\big\{t;\ J_\psi(t,X_t)=\psi(X_t)\big\},
$$
which is reminiscent of the graphical structure describing the optimizers in the deterministic mass transports problems studied by Brenier \cite{B1},  Gangbo-McCann \cite{G-M} and others.

\section{Stochastic Formulations} \label{sec:stochastic_formulations}\label{sec:stoch-form}
\subsection{Basic assumptions and notations}
We shall assume throughout the paper that the Lagrangian, $(t,x,u) \in \mathbb{R}^+ \times \mathbb{R}^d \times U \mapsto L(t,x,u)\in \mathbb{R}^+$, where $U\subset \R^d$, is a continuous function of time, position, and drift, uniformly continuous in $(t,x)$ (uniformly with respect to $u$) and is convex with respect to the drift, i.e.\ $u\mapsto L(t,x,u)$ is convex for all $(t,x)\in \R^+\times \R^d$.

In addition, we assume either that $U=\R^d$ and the Lagrangian $L$ is superlinear with respect to the drift and bounded from below, i.e. there is some $c>0$ and $p>1$ such that
\begin{equation}\label{eqn:L-assumption-1}
    c\big(|u|^p +1\big)\leq L(t,x,u),\,\,\,\,\mbox{for all}\,\,\,(t,x,u) \in \mathbb{R}^+ \times \mathbb{R}^d \times \mathbb{R}^d,
\end{equation}
or that $U$ is a bounded convex subset of $\R^d$. 
 Let $\Omega =  C(\R^+, \R^d)$ be the space of continuous paths from $\mathbb{R}^+$ to $\mathbb{R}^d$, $(X_t)_{t \in \mathbb{R}^+}$ be the canonical process, i.e. $X_t(\omega)=\omega(t)$\, for every $\omega\in \Omega$, and $\mathbb{F}=\{\mathcal{F}_t\}_{t\in \R^+}$ be the canonical filtration generated by $X$. Let $\mu,\,\nu$ be two probability measures on $\mathbb{R}^d$. The two (equivalent) formulations of the stochastic transport problem are the following:
\subsection{The weak stochastic formulation}\label{sec:weak_stochastic}
 We say that the triplet $(\mathbb{P},\beta,\alpha)$ belongs to $\mathcal{A}(\mu)$ if the following conditions hold:
\begin{enumerate}[label=\roman*)]
\item
$(\Omega,\mathbb{F},\mathbb{P})$ is a filtered probability space.
\item The initial distribution is given by $\mu$, i.e.\ $X_0\sim_\mathbb{P} \mu$ which means that ${X_0}_\# \mathbb{P}=\mu$.
\item The drift, $\beta:\R^+\times \Omega \rightarrow U$, is $\mathbb{F}$-progressively measurable and locally integrable, i.e.\ for each $\tau$, the map $\beta^\tau:\Omega \rightarrow L^1([0,\tau],U)$, which is given by $\beta^\tau(\omega)(t)=\beta_t(\omega)$ for every $t \in [0,\tau]$, satisfies $\beta^\tau$ is $\mathcal{F}_\tau-$measurable and
$
    \mathbb{E}^\mathbb{P}\big[\|\beta^\tau\|_{L^1([0,\tau],U)}\big]<+\infty.$
\item The process $W^\beta$ given by
$$W_t^\beta:=X_t-X_0-\int_0^t\beta_s\,ds,\,\,\,\mbox{for every}\,\,t\in \mathbb{R}^+,$$
is the standard Brownian motion, i.e.\ ${W^\beta}_{\#}\mathbb{P}$ is the Wiener measure on $\Omega$ with $W_0=0$. In other words, $X_t$ has the following semimartingale decomposition:
\begin{align}\label{eqn:weak_semimartingale_decomposition}
    X_t= X_0+\int_0^t \beta_s\,ds+W_t^\beta.
\end{align} 
\item $\alpha:\Omega \rightarrow \mathcal{M}(\R^+)$, where $\mathcal{M}(\R^+)$ is the space of measures on $\R^+$, is a randomized stopping time or equivalently, $A_t(\omega):=\alpha(\omega)([0,t])$ is increasing, right continuous, adapted to $\mathbb{F}$, with $A_0\geq 0$ and $\lim_{t\rightarrow \infty} A_t(\omega)=1$.  

\end{enumerate}

\noindent \hspace{0.1cm} Let  $(\alpha \ltimes \mathbb{P})(d\tau,d\omega)=\alpha(\omega)(d\tau)\mathbb{P}(d\omega)$ denote the measure on $\R^+\times \Omega$ corresponding to the variable $(\tau,\omega)$.
The constraint $X_\tau \sim_{\alpha \ltimes \mathbb{P}} \nu$ is equivalently defined by ${X_\tau}_\#(\alpha \ltimes \mathbb{P})=\nu$, i.e. we have
\begin{align}\label{eqn:target_weak_stochastic}
     \int_{\Omega}\int_{\R^+} g\big(\omega(\tau)\big)\alpha(\omega)(d\tau)\mathbb{P}(d\omega)=\int_{\R^d}g(y)\nu(dy),\,\,\mbox{for all}\,\,\,g\in C_b(\R^d).
\end{align}\\
Now, we let $\mathcal{A}(\mu,\nu) = \bigg\{(\mathbb{P},\beta,\alpha)\in \mathcal{A}(\mu)\,:\,X_\tau \sim_{\alpha \ltimes \mathbb{P}} \nu\bigg\}$. The cost, defined on $\mathcal{A}(\mu)$ with possible value of $+\infty$, is given by the following:
\begin{align}\label{eqn:cost_weak_stochastic}
    \mathcal{J}_L(\mathbb{P},\beta,\alpha)=&\ \mathbb{E}^{\alpha\ltimes \mathbb{P}}\Big[\int_0^\tau L\big(t,X_t,\beta_t\big)dt\Big]\\
    =&\ \int_{\Omega}\int_{\R^+} \int_0^{\tau} L\big(t,X_t(\omega),\beta_t(\omega)\big)dt\, \alpha(\omega)(d\tau) \mathbb{P}(d\omega).\nn 
\end{align} 

\noindent \hspace{0.1cm} We can state our primal stochastic transportation problem as the minimization of the stochastic transport cost $\mathcal{J}_L(\mathbb{P},\beta,\alpha)$ among all admissible $(\mathbb{P},\beta,\alpha)\in \mathcal{A}(\mu,\nu)$, i.e.
\begin{align}\label{eqn:primal_weak_stochastic}
    \mathcal{P}_{L}(\mu,\nu)  := \inf\bigg\{\mathcal{J}_L(\mathbb{P},\beta,\alpha)\,:\,\ {(\mathbb{P},\beta,\alpha)\in \mathcal{A}(\mu,\nu)}\bigg\}   
\end{align}
with the convention that $\mathcal{P}_{L}(\mu, \nu) =+\infty$ if $\mathcal{A}(\mu, \nu)=\emptyset$ (in fact, we will prove in Section \ref{sec:attainment-large-p} that, under some assumptions on $\mu$ and $\nu$, this set $\mathcal{A}(\mu,\nu)$ is non-empty).


%

\subsection{The convex stochastic formulation}


With this formulation, we seek to linearize the functional in (\ref{eqn:cost_weak_stochastic}) by considering probability measures \,$\tilde{\mathbb{P}}$\, on $\tilde{{\Omega}}:=\Omega\times \Omega$ (this idea is due to Haussmann \cite{haussmann1986existence} and, it is used later by Tan \& Touzi \cite{tan2013optimal}).
   In other words, let $(X,B)$ be the canonical process on $\tilde{{\Omega}}$ (i.e. $(X_t,B_t)(\omega,b)=(\omega(t),b(t))$, for every $(\omega,b) \in \tilde{{\Omega}}$) and let $\tilde{\mathbb{F}}$ be the corresponding canonical filtration. We now have a process
       \begin{equation} \label{eqn:W_convex_stochastic}
         \tilde{W}_t(\omega,b):={X}_t({\omega})-{X}_0(\omega)-{B}_t(b),
     \end{equation}
     which will play the role of $W^\beta$ in the definition of $\mathcal{A}(\mu)$. 
     We denote by $\tilde{\mathbb{P}}_X$ the projection onto the first component of $\Omega\times\Omega$, and $\tilde{\mathbb{P}}_B$ the projection onto the second component. We say that $(\tilde{\mathbb{P}},\tilde{\alpha})\in\tilde{\mathcal{A}}(\mu)$ if the following conditions hold:
 \begin{enumerate}[label=\roman*)]
     \item
     $(\tilde{{\Omega}},\tilde{\mathbb{F}},\tilde{\mathbb{P}})$ is a filtered probability space.
    \item We have ${X}_0\sim_{\tilde{\mathbb{P}}_X} \mu$ which means that ${{X}_0}_\#\tilde{\mathbb{P}}_X=\mu$.
    \item For $\tilde{\mathbb{P}}_B$ almost every $b$, $B$ is differentiable and $\tilde{\beta}={B}'$ is locally integrable, i.e.\ for each $\tau$, the map $\tilde{\beta}^{\,\tau}:\Omega \rightarrow L^1([0,\tau],U)$ defined by $\tilde{\beta}^\tau(b)(t)=\tilde{\beta}_t(b)$ satisfies
    $$
        \mathbb{E}^{\tilde{\mathbb{P}}}\big[\|\tilde{\beta}^\tau\|_{L^1([0,\tau],\R^d)}\big]<+\infty.
    $$
\item The process $\tilde{W}$ defined in (\ref{eqn:W_convex_stochastic}) is a standard Brownian motion, i.e.\ $\tilde{W}_{\#}\tilde{\mathbb{P}}$ is the Wiener measure on $\Omega$ with $\tilde{W}_0=0$. 
\item  $\tilde{\alpha}:\tilde{{\Omega}} \rightarrow \mathcal{M}(\R^+)$ is a randomized stopping time, equivalently $\tilde{A}_t(\omega,b):=\tilde{\alpha}(\omega,b)([0,t])$ is increasing, right continuous, adapted to $\tilde{\mathbb{F}}$, with $\tilde{A}_0\geq 0$ and $\lim_{t\rightarrow \infty} \tilde{A}_t(\omega,b)=1$. \\
\end{enumerate}

\noindent \hspace{0.1cm} We also define $\tilde{\alpha}\ltimes \tilde{\mathbb{P}}$ to be the associated probability measure on $\R^+\times \Omega\times \Omega$\, and \,$(\tilde{\alpha}\ltimes \tilde{\mathbb{P}})_{T,X}$ its projection onto the first two components.
The constraint $X_\tau \sim_{(\tilde{\alpha}\ltimes\tilde{\mathbb{P}})_{T,X}} \nu$ is similarly defined by
\begin{align}\label{eqn:target_convex_stochastic}
     \int_{\Omega}\int_{\Omega}\int_{\R^+} g\big(\omega(\tau)\big)\tilde{\alpha}(\omega,b)(d\tau)\tilde{\mathbb{P}}(d\omega,db)=\int_{\R^d}g(y)\nu(dy),\,\,\,\mbox{for all}\,\,\, g\in C_b(\R^d).
\end{align}
Now, we say that $(\tilde{\mathbb{P}},\tilde{\alpha})\in \tilde{\mathcal{A}}(\mu,\nu)$ if $(\tilde{\mathbb{P}},\tilde{\alpha})\in \tilde{\mathcal{A}}(\mu)$ and $X_\tau \sim_{(\tilde{\alpha}\ltimes\tilde{\mathbb{P}})_{T,X}} \nu$. The cost is similarly given by the following:
\begin{align}\label{eqn:cost}
    \tilde{\mathcal{J}}_L(\tilde{\mathbb{P}},\tilde{\alpha})=&\ \mathbb{E}^{\tilde{\alpha}\ltimes \tilde{\mathbb{P}}}\Big[\int_0^\tau L\big(t,{X}_t,\tilde{\beta}_t\big)dt\Big]\\
    =&\ \int_{{\Omega}}\int_\Omega\int_{\R^+} \int_0^{\tau} L\big(t,{X}_t({\omega}),\tilde{\beta}_t(b)\big)dt\, \tilde{\alpha}(\omega,b)(d\tau) \tilde{\mathbb{P}}(d\omega,db).\nn
\end{align}
 
\noindent \hspace{0.1cm} Then, we consider the following relaxation of (\ref{eqn:primal_weak_stochastic}) (again, we let  $\tilde{\mathcal{P}}_{L}(\mu, \nu) =+\infty$ if $\tilde{\mathcal{A}}(\mu, \nu)$ is empty):
\begin{equation}\label{eqn:primal_convex_stochastic}
    \tilde{\mathcal{P}}_{L}(\mu,\nu)  := \inf\Big\{\tilde{\mathcal{J}}_L(\tilde{\mathbb{P}},\tilde{\alpha})\,:\,\ (\tilde{\mathbb{P}},\tilde{\alpha})\in \tilde{\mathcal{A}}(\mu,\nu)\Big\}.
 \end{equation}
 
\noindent \hspace{0.1cm} We now show that problems (\ref{eqn:primal_convex_stochastic}) and (\ref{eqn:primal_weak_stochastic}) are equivalent in the sense that the two problems have the same minimal values. First, we show that the convex stochastic problem (\ref{eqn:primal_convex_stochastic})  is a relaxation of the weak stochastic formulation (\ref{eqn:primal_weak_stochastic}), and then show that there is a projection from the convex stochastic problem back onto the weak stochastic formulation that does not increase the cost. More precisely, we have the following:

 \begin{proposition} \label{prop:stochastic_relaxation} With the above notations, the following hold:
 \begin{enumerate}
\item  For every $(\mathbb{P},\beta,\alpha) \in \mathcal{A}(\mu)$, there exists $(\tilde{\mathbb{P}},\tilde{\alpha})\in \tilde{\mathcal{A}}(\mu)$ with $(\tilde{\alpha}\ltimes\tilde{\mathbb{P}})_{T,X}=\alpha\ltimes \mathbb{P}$ and $\mathcal{J}_L(\mathbb{P},\beta,\alpha)= \tilde{\mathcal{J}}_L(\tilde{\mathbb{P}},\tilde{\alpha})$.

\item Conversely, for every $(\tilde{\mathbb{P}},\tilde{\alpha}) \in \tilde{\mathcal{A}}(\mu)$, we can find $\alpha,\,\beta$ such that $(\tilde{\mathbb{P}}_X,\beta,{\alpha})\in \mathcal{A}(\mu)$ with $(\tilde{\alpha}\ltimes\tilde{\mathbb{P}})_{T,X}=\alpha\ltimes \tilde{\mathbb{P}}_X$ and \,$\mathcal{J}_L(\tilde{\mathbb{P}}_X,\beta,\alpha) \leq \tilde{\mathcal{J}}_L(\tilde{\mathbb{P}},\tilde{\alpha})$.

\item  In particular, we have $\mathcal{P}_L(\mu,\nu)=\tilde{\mathcal{P}}_L(\mu,\nu)$.
\end{enumerate}
 \end{proposition}
 \begin{proof}
 Take $(\mathbb{P},\beta,\alpha) \in \mathcal{A}(\mu)$ and define $B^\beta_t(\omega):=\int_0^t\beta_s(\omega)\,ds$.  We set $\tilde{\mathbb{P}}=(X,B^\beta)_{\#}\mathbb{P}$ and $\tilde{\alpha}(\omega,b)=\alpha(\omega)$.  We can easily check the five properties defining $\tilde{\mathcal{A}}(\mu)$:
 \begin{enumerate}[label=\roman*)]
     \item   $(\tilde{\Omega},\tilde{\mathbb{F}},\tilde{\mathbb{P}})$ is clearly a filtered probability space.
     \item Since $\tilde{\mathbb{P}}_X=\mathbb{P}$,  we are also dealing with the same initial distribution (${{X}_0}_\#\tilde{\mathbb{P}}_X={{X}_0}_{\#}\mathbb{P}=\mu$).
     \item We have that
     $$\mathbb{E}^{\tilde{\mathbb{P}}}\big[\|\tilde{\beta}^\tau\|_{L^1([0,\tau],\R^d)}\big]=\mathbb{E}^{{\mathbb{P}}}\big[\|{\beta}^\tau\|_{L^1([0,\tau],\R^d)}\big]<+\infty.
     $$
     \item Note that $\tilde{W}(\omega,B^\beta(\omega))=W^\beta(\omega)$ and therefore $\tilde{W}_{\#}\tilde{\mathbb{P}}={W^\beta}_{\#}\mathbb{P}$ is the Wiener measure on $\Omega$ with $\tilde{W}_0=0$.
     \item It is easy to check that $\tilde{\alpha}$ is still a randomized stopping time on the extended space.
\end{enumerate}
 On the other hand, we also have $(\tilde{\alpha}\ltimes \tilde{\mathbb{P}})_{T,X}=\alpha \ltimes \mathbb{P}$. From the definitions of $\mathcal{J}_L$ and $\tilde{\mathcal{J}}_L$, we have
\begin{align*}
    \tilde{\mathcal{J}}_L(\tilde{\mathbb{P}},\tilde{\alpha})=&\ \int_{{\Omega}}\int_\Omega\int_{\R^+} \int_0^\tau L\big(t,X_t(\omega),\tilde{\beta}_t(b)\big)dt\, {\alpha}(\omega)(d\tau) (X,B^\beta)_{\#}\mathbb{P}(d\omega,db)\\
    =&\  \int_{\Omega}\int_{\R^+} \int_0^\tau L\big(t,X_t(\omega),\beta_t(\omega)\big)dt\, \alpha(\omega)(d\tau)\mathbb{P}(d\omega)
\end{align*}
and so, it follows that $\mathcal{J}_L(\mathbb{P},\beta,\alpha)= \tilde{\mathcal{J}}_L(\tilde{\mathbb{P}},\tilde{\alpha})$.

For the second claim, take $(\tilde{\mathbb{P}},\tilde{\alpha})\in \tilde{\mathcal{A}}(\mu)$ and set $\mathbb{P}=\tilde{\mathbb{P}}_X$. We define $\alpha$ by disintegration in such a way that $(\tilde{\alpha}\ltimes \tilde{\mathbb{P}})_{T,X}=\alpha \ltimes \mathbb{P}$ and $\beta$ by the conditional expectation $\beta_t(\omega):=\mathbb{E}^{\tilde{\mathbb{P}}}\big[\tilde{\beta}_t\big| \tilde{\mathcal{F}}^X_t\big](\omega)$ (here, $\tilde{\mathcal{F}}^X_t$ is the filtration generated by process
$X$).  Again, we have the following five properties:
 \begin{enumerate}[label=\roman*)]
     \item  $(\Omega,\mathbb{F},\mathbb{P})$ is a filtered probability space.
     \item Since $\tilde{\mathbb{P}}_X=\mathbb{P}$, we still have the same initial distribution.
     \item By Jensen's inequality, we have that
     $$
         \mathbb{E}^{{\mathbb{P}}}\big[\|\mathbb{E}^{\tilde{\mathbb{P}}}\big[\tilde{\beta}^\tau\big|\tilde{\mathcal{F}}^X_t\big]\|_{L^1([0,\tau],\R^d)}\big]\leq \mathbb{E}^{\tilde{\mathbb{P}}}\big[\|\tilde{\beta}^\tau\|_{L^1([0,\tau],\R^d)}\big]<+\infty.
     $$
     \item From \cite[Theorem 4.3]{Wong}, we have that the following process is a standard Brownian motion:
 \begin{align*}
     W^\beta_t:=&X_t-X_0 -\int_0^t\mathbb{E}^{\tilde{\mathbb{P}}}\big[\tilde{\beta}_s\big| \tilde{\mathcal{F}}^X_s\big] ds.
 \end{align*}
     \item Finally, it is straightforward to verify that ${\alpha}$ is a randomized stopping time. 
\end{enumerate}
\noindent \hspace{0.1cm} On the other hand, using the definitions of $\mathbb{P}$, $\beta$ and $\alpha$ and Jensen's inequality, we get that
\begin{align*}
\mathcal{J}_L(\mathbb{P},\beta,\alpha) \leq \tilde{\mathcal{J}}_L(\tilde{\mathbb{P}},\tilde{\alpha}).
\end{align*}
Since we have identity of the distributions $\alpha\ltimes \mathbb{P}=(\tilde{\alpha}\ltimes \tilde{\mathbb{P}})_{T,X}$, it follows that the values of the primal weak stochastic and convex stochastic problems are equal. $\qedhere$
 \end{proof}

\noindent \hspace{0.1cm}
The convex problem \eqref{eqn:primal_convex_stochastic} satisfies a compactness property for a suitable topology on 
 the measures \,$\tilde{\alpha}\ltimes \tilde{\mathbb{P}}$\, on \,$\R^+\times \Omega\times \Omega$. We note that a simple truncation of the stopping time allows us to restrict to a compact domain in time and space.  We define the truncation for $(\mathbb{P},\beta,\alpha)\in \mathcal{A}(\mu)$ by considering the randomized stopping corresponding to $\tau \wedge \sup\{\tau;\,\tau\leq T,\,|X_\tau|\leq R\}$, i.e.\ $\alpha^{T,R}(\omega):={S^{T,R}_\omega}_\sharp \alpha$ where $S^{T,R}_\omega(\tau)=\tau \wedge \sup\{t;\ t\leq T,\,|X_t(\omega)|\leq R\}$. We set $S^{T,R}(\tau,w):=(S^{T,R}_w(\tau),w)$.
\begin{lemma}\label{lem:truncation}
	For any $(\mathbb{P},\beta,\alpha)\in \mathcal{A}(\mu)$, the truncation with \,$T,R\in \R^+$ satisfies $(\mathbb{P},\beta,\alpha^{T,R})\in \mathcal{A}(\mu)$ with $\alpha^{T,R}$ is supported on $[0,T]$ and 
	$|X_t(w)|\leq R$ almost surely.
	Furthermore,
	$$
		\lim_{T,R\rightarrow \infty} \mathcal{J}_L(\mathbb{P},\beta,\alpha^{T,R})=\mathcal{J}_L(\mathbb{P},\beta,\alpha), 
	$$
	and for all $g\in C_b(\R^d)$,
	$$
		\lim_{T,R\rightarrow \infty} \mathbb{E}^{\alpha^{T,R}\ltimes\mathbb{P}}\big[g(X_\tau)\big]=\mathbb{E}^{\alpha\ltimes\mathbb{P}}\big[g(X_\tau)\big].
	$$
\end{lemma}
\begin{proof}
 Since the range of $S^{T,R}$ is contained in $[0,T]\times \Omega$, it is clear that \,$\alpha^{T,R}$ is supported on $[0,T]$ and 
 $|X_t(w)|\leq R$ holds $(\alpha^{T,R}\ltimes \mathbb{P})$-almost surely. We then have that $(\mathbb{P},\beta,\alpha^{T,R})\in \mathcal{A}(\mu)$ since $S^{T,R}$ is continuous and maps $[0,\tau]\times \Omega$ into $[0,\tau]\times \Omega$, for all $\tau \in \mathbb{R}^+$, and leaves $\mathbb{P}$ and $\beta$ unchanged so that the decomposition (\ref{eqn:weak_semimartingale_decomposition}) holds.  Furthermore,
 it follows by the monotone convergence theorem that the limit of $\mathcal{J}_L(\mathbb{P},\beta,\alpha^{T,R})$ is $\mathcal{J}_L(\mathbb{P},\beta,\alpha)$ and that
 the end measure is also recovered. $\qedhere$
\end{proof}
 
 Before we have compactness we also need a truncation of the measure $\tilde{\mathbb{P}}$ so that weak limits do not lead to measures associated with badly behaved drifts after the end-time.
\begin{proposition}\label{prop:convex_compact}
    The set $\{\tilde{\alpha} \ltimes\tilde{\mathbb{P}} \,:\,(\tilde{\mathbb{P}},\tilde{\alpha}) \in \tilde{\mathcal{A}}(\mu)\}$ is convex. 
    Moreover, given any sequence $(\tilde{\mathbb{P}}^i,\tilde{\alpha}^i)\in \tilde{\mathcal{A}}(\mu,\nu^i)$ with $\nu^i \rightharpoonup \nu$, there is a 
$(\tilde{\mathbb{P}},\tilde{\alpha})\in \tilde{\mathcal{A}}(\mu,\nu)$ such that 
$\tilde{\mathcal{J}}_L(\tilde{\mathbb{P}},\tilde{\alpha}) \leq \liminf_{i\rightarrow \infty} \tilde{\mathcal{J}}_L(\tilde{\mathbb{P}}^{i},\tilde{\alpha}^{i})$.
\end{proposition}
\begin{proof}
The convexity follows immediately from \cite[Corollary III.2.8]{Jacod}. Now,
for each $\tau\in \R^+$, let us define the map $S^\tau:\Omega\times \Omega\rightarrow\Omega\times \Omega$ by
\begin{align}\label{eqn:truncation_map}
    S^\tau(\omega,b):=(\hat{\omega}^{\tau},\hat{b}^{\tau}),
\end{align}
where
$$
    \hat{\omega}^{\tau}(t) := \omega(t)-b(t)+b(\tau\wedge t)
\,\,\,\,\mbox{and}\,\,\,\,
    \hat{b}^{\tau}(t) := b(\tau\wedge t).
$$
Notice that, for every $\tau \in \mathbb{R}^+$, the map $S^\tau$ is invariant for $\tilde{W}$ in the sense that
$$
    \tilde{W}_t\big(S^\tau(\omega,b)\big)=\hat{\omega}^{\tau}(t)-\hat{\omega}^{\tau}(0)-\hat{b}^{\tau}(t)=\omega(t)-\omega(0)-b(t)=\tilde{W}_t(\omega,b).
$$
We define the pair $(\tilde{\mathbb{Q}}^i,\tilde{\gamma}^i)$ by duality such that
$$
    \int_{\Omega}\int_\Omega\int_{\R^+} H(\tau,\omega,b) \tilde{\gamma}^i(\omega,b)(d\tau)\tilde{\mathbb{Q}}^i(d\omega,db)=\int_{{\Omega}}\int_\Omega\int_{\R^+} H(\tau,\hat{\omega}^\tau,\hat{b}^\tau) \tilde{\alpha}^i(\omega,b)(d\tau)\tilde{\mathbb{P}}^i(d\omega,db),
$$
for all $H \in C_b(\mathbb{R}^+ \times \Omega \times \Omega)$. In particular, the truncated measure $\tilde{\mathbb{Q}}^i$ satisfies
\begin{align*}
    \int_{{\Omega}}\int_\Omega F(\omega,b) \tilde{\mathbb{Q}}^i(d\omega,db)=\int_{{\Omega}}\int_\Omega\int_{\R^+} F(\hat{\omega}^\tau,\hat{b}^\tau) \tilde{\alpha}^i(\omega,b)(d\tau)\tilde{\mathbb{P}}^i(d\omega,db),\,\mbox{for all}\,F \in C_b(\Omega \times \Omega).
\end{align*}
Now, we will show that $(\tilde{\mathbb{Q}}^i,\tilde{\gamma}^i)\in \tilde{\mathcal{A}}(\mu,\nu)$ with the same cost. First, it is not difficult to check that the properties (i),\,(ii),\,(iii), (iv) and (v) in the definition of $\tilde{\mathcal{A}}(\mu)$ hold. In fact, for all $G \in C_b(\Omega)$, we have
\begin{align*}
    \int_{{\Omega}} G(\omega)\tilde{W}_\# \tilde{\mathbb{Q}}^i(d\omega)=&\ \int_{\tilde{\Omega}}G\big(\tilde{W}(\omega,b)\big)\tilde{\mathbb{Q}}^i(d\omega,db)\\
    =&\ \int_{{\Omega}}\int_\Omega\int_{\R^+} G\big(\tilde{W}(\hat{\omega}^\tau,\hat{b}^\tau)\big)\tilde{\alpha}^i(\omega,b)(d\tau)\tilde{\mathbb{P}}^i(d\omega,db)\\
    =&\ \int_{{\Omega}}\int_\Omega G\big(\tilde{W}(\omega,b)\big)\tilde{\mathbb{P}}^i(d\omega,db).
\end{align*}
Thus, $\tilde{W}_\# \tilde{\mathbb{Q}}^i=\tilde{W}_\# \tilde{\mathbb{P}}^i$. On the other hand,
we have
\begin{align*}
     \int_{\Omega}\int_\Omega\int_{\R^+}f(\omega(\tau))\tilde{\gamma}^i(\omega,b)(d\tau)\tilde{\mathbb{Q}}^i(d\omega,db)
    =&\ \int_{{\Omega}}\int_\Omega\int_{\R^+}f(\hat{\omega}^\tau(\tau)\big)\tilde{\alpha}^i(\omega,b)(d\tau)\tilde{\mathbb{P}}^i(d\omega,db)\\
    =&\  \int_{{\Omega}}\int_\Omega\int_{\R^+}f(\omega(\tau))\tilde{\alpha}^i(\omega,b)(d\tau)\tilde{\mathbb{P}}^i(d\omega,db)\\
    =& \int f \, \mathrm{d}\nu,\,\,\,\mbox{for all}\,\,f \in C_b(\mathbb{R}^d).
\end{align*}
This implies that $X_\tau \sim_{(\tilde{\gamma}^i\ltimes\tilde{\mathbb{Q}}^i)_{T,X}} \nu$. Similarly, we see that the cost does not change since
\begin{align*}
    &\ \int_{\Omega}\int_\Omega\int_{\R^+}\int_0^\tau L(t,X_t(\omega),\tilde{\beta}_t(b))\mathrm{d}t\,\tilde{\gamma}^i(\omega,b)(d\tau)\tilde{\mathbb{Q}}^i(d\omega,db)\\
    =&\ \int_{{\Omega}}\int_\Omega\int_{\R^+}\int_0^\tau L(t,X_t(\hat{\omega}^\tau),\tilde{\beta}_t(\hat{b}^\tau))\mathrm{d}t\,\tilde{\alpha}^i(\omega,b)(d\tau)\tilde{\mathbb{P}}^i(d\omega,db)\\
    =&\  \int_{{\Omega}}\int_\Omega\int_{\R^+}\int_0^\tau L(t,X_t(\omega),\tilde{\beta}_t(b))\mathrm{d}t\,\tilde{\alpha}^i(\omega,b)(d\tau)\tilde{\mathbb{P}}^i(d\omega,db).
\end{align*}
\noindent The advantage of this truncation is that we now have that
\begin{align*}
\int_{{\Omega}}\int_\Omega\int_{\R^+} |\tilde{\beta}_t(b)|^pdt\, \tilde{\mathbb{Q}}^i(d\omega,db)=&\ \int_{{\Omega}}\int_\Omega\int_{\R^+}\int_{\R^+} |\tilde{\beta}_t(\hat{b}^\tau)|^pdt\, \tilde{\alpha}^i(\omega,b)(d\tau)\tilde{\mathbb{P}}^i(d\omega,db)\\
=&\  \int_{{\Omega}}\int_\Omega\int_{\R^+} \int_0^\tau |\tilde{\beta}_t(b)|^pdt\, \tilde{\alpha}^i(\omega,b)(d\tau)\tilde{\mathbb{P}}^i(d\omega,db),\\
\leq&\ c^{-1} \tilde{\mathcal{J}}_L(\tilde{\mathbb{P}}^i,\tilde{\alpha}^i),
\end{align*}
where the last inequality comes from the fact that the Lagrangian $L$ is superlinear, or in the case that $U$ is bounded this holds for any $p$ with $c=\inf_{(t,x,u)\in \R^+\times \R^d\times U} L(t,x,u)/|u|^p$. Furthermore, we have 
  \begin{align*}
      &\ \int_\Omega\int_{{\Omega}}\tilde{\gamma}^i(\omega,b)([R,+\infty))\tilde{\mathbb{Q}}^i(d\omega,db)\\
      \leq&\ \frac{1}{c R}\int_{\tilde{\Omega}}\int_{R}^\infty \int_0^{\tau} L\big(t,X_t({\omega}),\tilde{\beta}_t(b)\big)dt\, \tilde{\gamma}^i(\omega,b)(d\tau)\tilde{\mathbb{Q}}^i(d\omega,db)\\
      \leq&\ \frac{1}{c R} \tilde{\mathcal{J}}_L(\tilde{\mathbb{Q}}^i,\tilde{\gamma}^i).
\end{align*}

\noindent  Then, by \cite{meyer1984tightness,zheng1985tightness}, we infer that the sequence
$\tilde{\gamma}^i\ltimes \tilde{\mathbb{Q}}^i$ is tight
  and so, there is some $(\tilde{\mathbb{P}},\tilde{\alpha}) \in \tilde{\mathcal{A}}(\mu)$ such that
  $(\tilde{\gamma}^{i_k}\ltimes \tilde{\mathbb{Q}}^{i_k})\rightharpoonup(\tilde{\alpha}\ltimes \tilde{\mathbb{P}})$. In particular, this implies that $(\tilde{\gamma}^{i_k}\ltimes \tilde{\mathbb{Q}}^{i_k})_{T,X}\rightharpoonup(\tilde{\alpha}\ltimes \tilde{\mathbb{P}})_{T,X}$ 
 and $X_\tau \sim_{(\tilde{\alpha}\ltimes\tilde{\mathbb{P}})_{T,X}}\nu$.

\noindent \hspace{0.1cm} Now, we define
$$
U_{\varepsilon,\delta,\tau}(\omega)=\begin{cases} 1 & \mbox{if}\,\,\sup_{s \in \left[0,\tau-\varepsilon\right],|t-s|\leq \varepsilon} |X_s(w) - X_t(w)|< \delta,\\ 0 & \mbox{otherwise}. \end{cases}
$$
Using the fact that $ L$ is uniformly continuous in $(t,x)$ and by Jensen's inequality, it is not difficult to check that, when $U_{\varepsilon,\delta,\tau}(\omega)=1$, we have
 $$ \int_0^{\tau -  \varepsilon} L\Big(s,X_s(\omega),\frac{1}{\varepsilon}\int_s^{s + \varepsilon} \tilde{\beta}_t(b)\,dt\Big)ds \leq  \int_0^{\tau}  L\big(t,X_t(\omega),\tilde{\beta}_t(b)\big)\, dt + C(\varepsilon + \delta)\tau,$$
 where $C$ is the modulus of uniform continuity of the Lagrangian $L$ with respect to $t$ and $x$. Then, we get
 \begin{align*}
 \tilde{\mathcal{J}}_L(\tilde{\mathbb{Q}}^i,\tilde{\gamma}^i)
 \geq&\ \mathbb{E}^{\tilde{\gamma}^i\ltimes \tilde{\mathbb{Q}}^i} \Big[U_{\varepsilon,\delta,\tau}\int_0^{(\tau-\varepsilon)\vee 0} L\big(s,X_s,\frac{1}{\varepsilon}(B_{s + \varepsilon}(b) - B_s(b))\big)ds\Big] - C(\varepsilon + \delta).
 \end{align*}
The map $(\tau,w,b) \mapsto  U_{\varepsilon,\delta,\tau}(w)\int_0^{(\tau - \varepsilon)\vee 0} L(s,X_s(w),\frac{1}{\varepsilon}(B_{s + \varepsilon}(b) - B_s(b)))\mathrm{d}s$\, is lower semi-continuous and $\tilde{\gamma}^i\ltimes \tilde{\mathbb{Q}}^i \rightharpoonup  \tilde{\alpha}\ltimes \tilde{\mathbb{P}}$. So, passing to the limit when $i \to +\infty$, we get
 $$
     \liminf_{i \to \infty}\tilde{\mathcal{J}}_L(\tilde{\mathbb{Q}}^i,\tilde{\gamma}^i)\geq \mathbb{E}^{\tilde{\alpha}\ltimes \tilde{\mathbb{P}}} \Big[U_{\varepsilon,\delta,\tau}\int_0^{(\tau-\varepsilon)\vee 0} L\big(s,X_s,\frac{1}{\varepsilon}(B_{s + \varepsilon}(b) - B_s(b))\big)ds\Big]- C(\varepsilon +\delta).
     $$ 
  Letting $\varepsilon \to 0$ and using Fatou's Lemma as well as the continuity of $C$, we infer that
   $$\liminf_{i \to \infty}\tilde{\mathcal{J}}_L(\tilde{\mathbb{Q}}^i,\tilde{\gamma}^i) \geq \tilde{\mathcal{J}}_L(\tilde{\mathbb{P}},\tilde{\alpha}) - C(\delta),
   $$
   which completes the proof since $\delta>0$ is arbitrary and, thanks to the fact that $\tilde{\mathcal{J}}_L(\tilde{\mathbb{P}}^i,\tilde{\alpha}^i)=\tilde{\mathcal{J}}_L(\tilde{\mathbb{Q}}^i,\tilde{\gamma}^i)$, for every $i \in \mathbb{N}$. $\qedhere$
    \end{proof}
    As a consequence, we get
\begin{corollary}\label{Cor:ls.c. and convex}
The problem \eqref{eqn:primal_weak_stochastic} reaches a minimum. Moreover, the map $\nu \mapsto {\mathcal{P}}_L(\mu,\nu)$ is convex and lower semi-continuous.
\end{corollary}
\begin{proof}
Let $({\mathbb{P}}^i,{\alpha}^i)$ be a minimizing sequence in Problem \eqref{eqn:primal_weak_stochastic} and $(\tilde{\mathbb{P}}^i,\tilde{\alpha}^i)$ be the corresponding measures such that $\tilde{\mathcal{J}}_L(\tilde{\mathbb{P}}^i,\tilde{\alpha}^i)={\mathcal{J}}_L({\mathbb{P}}^i,{\alpha}^i)$ (see Proposition \ref{prop:stochastic_relaxation}). By Proposition \ref{prop:convex_compact}, there is a $(\tilde{\mathbb{P}},\tilde{\alpha})$ such that $\liminf_{i \to \infty}\tilde{\mathcal{J}}_L(\tilde{\mathbb{P}}^i,\tilde{\alpha}^i) \geq \tilde{\mathcal{J}}_L(\tilde{\mathbb{P}},\tilde{\alpha})$.
This yields that the projection $(\mathbb{P},\beta,\alpha)$ of $(\tilde{\mathbb{P}},\tilde{\alpha})$ (see again Proposition \ref{prop:stochastic_relaxation}) is a minimizer for (\ref{eqn:primal_weak_stochastic}). The second statement can be proved in a similar way using Propositions \ref{prop:stochastic_relaxation} \& \ref{prop:convex_compact}. $\qedhere$
\end{proof}
\section{Duality and dynamic programming} \label{sec:duality}
In this section, we verify a duality principle for the stochastic transportation problem \eqref{eqn:primal_weak_stochastic}. More precisely, we consider the following maximization principle:
\begin{equation} \label{eqn:dual_stochastic}
    \mathcal{D}_L(\mu,\nu):=\sup_{(\phi,\psi)\in \mathcal{Z}_L}\Big\{\int_{\R^d}\psi(y)\nu(dy)-\int_{\R^d}\phi(x)\mu(dx)\Big\},
\end{equation}
where $\mathcal{Z}_L$ is the set of functions $(\phi,\psi)$ on $\R^d$ with $\psi\in C_b(\R^d)$, continuous and bounded, and $\phi\in B_\mu(\R^d)$, measurable w.r.t. the Borel $\sigma-$field on $\mathbb{R}^d$ completed by $\mu$, such that the following holds:
\begin{align*}
    \phi(x)\geq \mathbb{E}^{\alpha\ltimes\mathbb{P}}\Big[\psi(X_\tau)-\int_0^\tau L(t,X_t,\beta_t)dt\Big],
\end{align*}
for all $(\mathbb{P},\beta,\alpha)\in \mathcal{A}(\delta_x)$, $\mu-$ a.e. $x \in \mathbb{R}^d$.

We can characterize the dual problem further using the dynamic programming principle. We first define a translation map $T^{t,x}:\Omega\rightarrow C([t,+\infty),\R^d)$, where
$$
    T^{t,x}(\omega)(s)=\omega(s-t)+x.
$$
Given $\psi\in C_b(\R^d)$, we introduce the function 
$$
    J_\psi(t,x):=\sup_{(\mathbb{P},\beta,\alpha)\in \mathcal{A}_t(\delta_x)} \mathbb{E}^{\alpha\ltimes \mathbb{P}}\Big[\psi(X_{\tau})-\int_t^{\tau} L(s,X_{s},\beta_{s})ds\Big]
$$
where
$$\mathcal{A}_t(\delta_x):=\{T^{t,0}_\#(\mathbb{P},\beta,\alpha);(\mathbb{P},\beta,\alpha) \in \mathcal{A}(\delta_x)\}.$$
Here, the pushforward on $\mathbb{P}$ is as a measure and on $\beta$ and $\alpha$ are given respectively by $\beta({T^{t,0}}^{-1}(\omega))$ and $\alpha({T^{t,0}}^{-1}(\omega))$ for all $\omega\in C([t,+\infty),\R^d)$. Moreover, we note that
$$
    J_\psi(t,x):=\sup_{(\tilde{\mathbb{P}},\tilde{\alpha})\in \tilde{\mathcal{A}}_t(\delta_x)} \mathbb{E}^{\tilde{\alpha}\ltimes \tilde{\mathbb{P}}}\Big[\psi(X_{\tau})-\int_t^{\tau} L(s,X_{s},\tilde{\beta}_{s})ds\Big]
$$
where 
$$\tilde{\mathcal{A}}_t(\delta_x):=\{\tilde{T}^{t,0}_\#(\tilde{\mathbb{P}},\tilde{\alpha});(\tilde{\mathbb{P}},\tilde{\alpha}) \in \tilde{\mathcal{A}}(\delta_x)\}$$
with
$$\tilde{T}^{t,0}(w,b)(s)=(w(s-t),b(s-t))$$
and $$\tilde{T}^{t,0}_\#\tilde{\alpha}:=\tilde{\alpha}({\tilde{T}^{t,0}}^{-1}(\omega,b))\,\,\,\mbox{for all}\,\,\,(\omega,b)\in C([t,+\infty),\R^{d} \times \mathbb{R}^d).$$
For any $(\phi,\psi) \in \mathcal{Z}_L$, we have that $J_\psi(0,x)\leq \phi(x)$ for $\mu-$a.e. $x\in \R^d$.
So, the dual problem (\ref{eqn:dual_stochastic}) becomes
$$
D_L(\mu,\nu)=\sup_{\psi\in C_b(\R^d)}\Big\{\int_{\R^d} \psi(y)\,\nu(dy) - \int_{\R^d} J_\psi(0,x)\, \mu(dx)\Big\}.
$$
We first prove a Lemma that verifies a basic level of regularity for $J_\psi$ needed in the proof of Theorem \ref{thm:stochastic_duality}.  Such a result is standard, see \cite{tan2013optimal} for a setting similar to ours.

\begin{lemma}\label{lem:measurability}
    We suppose that $\psi \in C_b(\mathbb{R}^d)$.
    Then, the map
    $(t,x)\mapsto J_\psi(t,x)$ is lower semi-continuous and bounded from below. In addition,
    we have
    \begin{align}\label{eqn:equality}
        \int_{\R^d}J_\psi(0,x)\mu(dx) = \sup_{(\mathbb{P},\beta,\alpha)\in \mathcal{A}(\mu)}\Big\{\mathbb{E}^{\alpha\ltimes\mathbb{P}} \Big[\psi(X_\tau)   -\int_0^\tau L(t,X_t,\beta_t)dt\Big]
  \Big\}.
    \end{align}
\end{lemma}
\begin{proof}
    We first note that we can express $J_\psi(t,x)$ as
    \begin{align*}
        J_\psi(t,x)=&\ \sup_{(\mathbb{P},\beta,\alpha)\in \mathcal{A}(\delta_0)}\Big\{\mathbb{E}^{T^{t,x}_\# (\alpha \ltimes \mathbb{P})}\Big[\psi(X_\tau)-\int_t^\tau L(s,X_s,T^{t,x}_\#\beta_s)ds\Big]\Big\}\\
        =&\ \sup_{(\mathbb{P},\beta,\alpha)\in \mathcal{A}(\delta_0)}\Big\{\mathbb{E}^{\alpha \ltimes \mathbb{P}}\Big[\psi(X_\tau+x)-\int_0^\tau L(s+t,X_s+x,\beta_s)ds\Big]\Big\}.
    \end{align*}
    For each $(\mathbb{P},\beta,\alpha)\in \mathcal{A}(\delta_0)$, this defines a continuous function of $t$ and $x$ by the continuity of the translation maps $T^{t,x}$, the uniform continuity of $L$ and the bound on $\psi$, and so $J_\psi(t,x)$ is the supremum over all these functions, making it lower semi-continuous. Clearly, $J_\psi(t,x)$ is bounded below by $\psi(x)$ (we also note that $J_\psi(t,x)$ is bounded above by $\sup_x\psi$).
   
    To prove (\ref{eqn:equality}), we first note that any $({\mathbb{P}},\beta ,{\alpha})\in {\mathcal{A}}(\mu)$ disintegrates w.r.t.\ $\mu$ by the map $S(\tau,\omega)=\omega(0)$ and yields, using the definition of $J_\psi$, that
    $$
        \int_{\R^d}J_\psi(0,x)\mu(dx) \geq \mathbb{E}^{{\alpha}\ltimes{\mathbb{P}}} \Big[\psi(X_\tau)   -\int_0^\tau L(t,X_t,{\beta}_t)dt\Big] .
    $$
    Furthermore, for each $\varepsilon>0$, the set
    $$
        \Big\{(x,{\alpha}\ltimes {\mathbb{P}});\ (\mathbb{P},\beta,\alpha)\in \mathcal{A}(\delta_x),\  \mathbb{E}^{\alpha \ltimes \mathbb{P}}\Big[\psi(X_\tau)-\int_0^\tau L(s,X_s,\beta_s)ds\Big]\geq J_\psi(0,x)-\varepsilon\Big\}
    $$
    is non-empty for $\mu-$a.e.\ $x$ 
    and
    closed by Propositions \ref{prop:stochastic_relaxation} \&
    \ref{prop:convex_compact} and thanks to the lower semi-continuity of $J_\psi$. Thus, by the Kuratowski and Ryll-Nardzewski selection theorem, we may find a $\mu-$Borel measurable map ${{Q}}(x)={\alpha}^x\ltimes {\mathbb{P}}^x$ with $(\mathbb{P}^x,\beta^x,\alpha^x)\in \mathcal{A}(\delta_x)$ and $\  \mathbb{E}^{\alpha^x \ltimes \mathbb{P}^x}\Big[\psi(X_\tau)-\int_0^\tau L(s,X_s,\beta^x_s)ds\Big]\geq J_\psi(0,x)-\varepsilon$, for $\mu-$a.e.\ $x$. Therefore, we may consider $({\mathbb{P}},\beta,{\alpha})\in {\mathcal{A}}(\mu)$ such that ${\alpha}\ltimes {\mathbb{P}}$ disintegrates w.r.t.\ $\mu$ as ${Q}(x)$, i.e.\ ${\alpha} \ltimes {\mathbb{P}} = {Q}\ltimes \mu$. Thus, we have
    $$
 \int_{\R^d} J_\psi(0,x)\,\mu(dx) -\varepsilon \leq \mathbb{E}^{{\alpha}\ltimes{\mathbb{P}}} \Big[\psi(X_\tau)   -\int_0^\tau L(t,X_t,{\beta}_t)dt\Big],
    $$\\
    which proves the equality (\ref{eqn:equality}) when $\varepsilon \to 0$. $\qedhere$
\end{proof}


 The scheme to prove duality is by now standard in convex analysis (we refer, for instance, to \cite{ekeland1999convex}). The idea is that if $\nu\mapsto V(\nu):=\mathcal{P}_L(\mu,\nu)$ is l.s.c.\ and convex, then $V^{\star\star}=V$, where $V^\star$ denotes the Legendre-Fenchel transform of $V$. The value $V(\nu)$ is the minimal value of our stochastic primal problem (\ref{eqn:primal_weak_stochastic}), while $V^{\star\star}(\nu)$ is identified with the maximal value of the dual problem (\ref{eqn:dual_stochastic}). This strategy was already used by many authors to establish dual principles for various optimal transport problems that do not fit in the Monge-Kantorovich theory (see \cite{mikami2008optimal,tan2013optimal}).

\begin{theorem} \label{thm:stochastic_duality}
The following equality holds:
$$\mathcal{P}_L(\mu,\nu)=D_L(\mu,\nu).$$
\end{theorem} 
\begin{proof} 
First, it is clear that the maximal value of the dual problem is less than or equal to the minimal value of the primal one, since, for all admissible $(\mathbb{P},\beta,\alpha) \in \mathcal{A}(\mu,\nu)$ and $(\phi,\psi)\in \mathcal{Z}_L$, we have
\begin{align*}
    \int_{\R^d} \psi(y)\,\nu(dy)-\int_{\mathbb{R}^d}\phi(x)\,\mu(dx)=&\ \mathbb{E}^{\alpha\ltimes \mathbb{P}}\Big[\psi(X_\tau)-\phi(X_0)\Big]
    \leq \ \mathbb{E}^{\alpha\ltimes \mathbb{P}}\Big[\int_0^\tau L(t,X_t,\beta_t)dt
    \Big]
\end{align*}
and so, it follows that $\mathcal{D}_L(\mu,\nu)\leq \mathcal{P}_L(\mu,\nu)$.

For the other direction, we have that $V(\nu)$ is convex and lower semi-continuous (see Corollary \ref{Cor:ls.c. and convex}). Then, we have
  $$
      V(\nu) =V^{\star \star} (\nu)= \sup_{\psi \in C_b(\mathbb{R}^d)}\Big\{\int_{\mathbb{R}^d} \psi(x) \nu(dx) -V^{\star}(\psi)\Big\}$$
  where
   \begin{align*}
       V^{\star}(\psi)=&\ \sup_{\nu \in \mathcal{M}(\mathbb{R}^d)}\Big\{\int_{\mathbb{R}^d} \psi(x) \nu(dx) - V(\nu)\Big\}\\
       =&\ \sup_{(\mathbb{P},\beta,\alpha) \in \mathcal{A}(\mu)}\Big\{\mathbb{E}^{\alpha\ltimes\mathbb{P}} \Big[\psi(X_\tau)   -\int_0^\tau L(t,X_t,\beta_t)dt\Big]
  \Big\}.
  \end{align*}
Note that we used that if $\nu$ is not a probability measure then $\mathcal{A}(\mu, \nu)=\emptyset$ and $V(\nu) =+\infty$, which does not affect the supremum in the definition of $V^\star$.
\noindent We then have by Lemma \ref{lem:measurability} that
$$
    V^*(\psi)=\int_{\R^d}J_\psi(0,x)\mu(dx)
$$
and
$$
   \mathcal{P}_L(\mu,\nu)=V^{**}(\nu)= \sup_{\psi\in C_b(\R^d)}\Big\{\int_{\R^d} \psi(y)\nu(dy)-\int_{\R^d} J_\psi(0,x)\mu(dx)\Big\}.
$$
By setting $\phi(x)=J_\psi(0,x)$, we have $(\phi,\psi)\in \mathcal{Z}_L$ and we have shown the reverse inequality, that is
$$
    \mathcal{D}_L(\mu,\nu)\geq \mathcal{P}_L(\mu,\nu),
$$
which completes the proof. $\qedhere$

\end{proof}


 
On the other hand, by the standard dynamic programming principle, we get that $J_\psi$ is a viscosity solution of the following dynamic programming equation (see, for instance, \cite{tan2013optimal}):
$$\min\Big\{J_\psi(t,x) - \psi(x),-\partial_t J_\psi(t,x) - \frac{1}{2} \Delta J_\psi(t,x) - H\big(t,x,\nabla J_\psi(t,x)\big)\Big\}= 0,$$
where the Hamiltonian $H$ is given by the Legendre dual of $L$, i.e. we have
\begin{align*}
 H(t, x, z) = \sup_{ u \in U} \big[ z \cdot u - L (t, x, u)\big],\,\,\mbox{for all}\,\,(t,x,z)\in\mathbb{R}^+ \times \mathbb{R}^d \times \mathbb{R}^d.
\end{align*}
This viscosity solution $J_\psi$
can be viewed as the minimal supersolution, given by the infimum of the smooth supersolutions of 
\begin{equation}\label{HJBESec.4}
\begin{cases}
\partial_t J(t,x) +  \frac{1}{2} \Delta J(t,x) + H\big(t,x,\nabla J(t,x)\big) \leq 0 &\ \mbox{in}\,\,\ \mathbb{R}^+ \times \mathbb{R}^d,\\
\psi(x) \leq J(t,x)   &\ \mbox{on}\,\,\ \mathbb{R}^+ \times\mathbb{R}^d.
\end{cases}
\end{equation}

Finally, we get the following:
\begin{proposition}\label{prop:dual_equivalence}
We have\\

\noindent $\mathcal{D}_L(\mu,\nu) = \sup\limits_{\psi\in C_b(\R^d)}\Big\{\int_{\R^d} \psi(y)\,\nu(dy) - \int_{\R^d} J_\psi(0,x)\, \mu(dx)\Big\}$
$$\qquad\,\,\;\,\,\;\,\,\,\,={\small{  \sup_{\psi \in C_b(\mathbb{R}^d),\,J\in C^{1,2}_b(\mathbb{R}^+ \times \mathbb{R}^d)}\Big\{\int_{\R^d} \psi(y)\nu(dy) - \int_{\R^d} J(0,x) \mu(dx);\  (\psi,J)\  \, \mbox{satisfies}\ \, \eqref{HJBESec.4}\Big\}}}.$$
\end{proposition}
\section{Eulerian formulations} \label{sec:Eulerian}
In this section, we express two Eulerian formulations for (\ref{eqn:primal_weak_stochastic}); the strong Eulerian formulation poses the problem with a velocity field and the solution to a Fokker-Planck equation with stopping, while the convex Eulerian formulation poses the problem over phase-space distributions satisfying a convex set of inequalities. We prove the later to be equivalent to (\ref{eqn:primal_weak_stochastic}) by embedding the stochastic formulation and showing a weak duality inequality. We then prove that when the drift is uniformly bounded, the strong Eulerian formulation is also equivalent by Sobolev estimates. We note that these estimates
will be used to prove dual attainment in the next section.
\subsection{Strong and convex Eulerian formulations}
First consider strong Eulerian formulation.
We say that $(m,v,\rho)$ belongs to $\mathcal{E}(\mu)$ if the following holds:
\begin{itemize}
    \item $m\in L^1(\R^+  \times \R^d) \cap L^2_{loc}(\R^+, H^1(\R^d))$ with $||m_t||_{L^1(\R^d)} \leq 1$, $v\in L^\infty(\R^+\times \R^d,\R^d)$
    \item $\rho$ is a probability measure on $\R^+\times \R^d$.
    \item
    For all smooth $\phi$ with compact support in $\R^+\times \R^d$,
    we have
    \begin{align}\label{eqn:strong_Eulerian_evolution}
    &\ \int_{\R^+}\int_{\R^d}\bigg[\partial_t \phi(t,x)\,m(t,x)+\nabla \phi(t,x)\cdot\bigg(v(t,x)\,m(t,x)-\frac{1}{2}\nabla m(t,x)\bigg)\bigg]dxdt\nn\\
        &\ \qquad\qquad\,\,= \int_{\R^+}\int_{\R^d} \phi(\tau,y)\rho(d\tau,dy)-\int_{\R^d}\phi(0,x)\mu(dx).
    \end{align}
\end{itemize}

\noindent \hspace{0.1cm} Then, we consider the following problem:
\begin{align}
    \mathcal{P}_L^{\mathcal{E}}(\mu,\nu)=\inf_{(m,v,\rho)\in \mathcal{E}(\mu)}\Big\{\int_{\R^+}\int_{\R^d} L\big(t,x,v(t,x)\big)m(t,x)\,dx dt\,:\,\ \int_{\R^+} \rho(d\tau,\cdot)=\nu\Big\}.
\end{align}

\noindent \hspace{0.1cm} Now, let us introduce the following convex Eulerian formulation:
\begin{definition}
We say that  $(\eta,\rho) \in \tilde{\mathcal{E}}(\mu)$ if the following holds:
\begin{itemize}
    \item $\eta$ is a measurable map from $\R^+$ to sub-probability measures on $\R^d\times U$,
    \item $\rho$ is a probability measure on $\R^+\times \R^d$,
    \item The following equation holds for all smooth $\phi$ with $\phi$, $\partial_t\phi$, $\nabla \phi$ and $\Delta \phi$ uniformly bounded on $\R^+\times \R^d$:
    \begin{align}\label{eqn:convex_Eulerian_evolution}
        &\ \int_{\R^+}\int_{\R^d}\int_{U}\Big[\partial_t \phi(t,x)+\frac{1}{2}\Delta \phi(t,x)+\nabla \phi(t,x)\cdot u\Big]  \eta(dt,dx,du)\nn\\
        &  \qquad\,\,=    \ \int_{\R^+}\int_{\R^d} \phi(\tau,y)\rho(d\tau,dy)-\int_{\R^d}\phi(0,x)\mu(dx).
    \end{align}
\end{itemize}
\end{definition}
\noindent \hspace{0.1cm} We let
\begin{align}\label{eqn:Eulerian_primal_convex}
    \mathcal{P}_L^{\tilde{\mathcal{E}}}(\mu,\nu)=\inf_{(\eta,\rho)\in \tilde{\mathcal{E}}(\mu)}\Big\{\int_{\R^+}\int_{\R^d}\int_{U} L(t,x,u) \eta(dt,dx, du)\,:\,\ \int_{\R^+} \rho(d\tau,\cdot)=\nu\Big\}.
\end{align}
 
\noindent It is clear that $\mathcal{P}_L^{\tilde{\mathcal{E}}}(\mu,\nu)\leq \mathcal{P}_L^{\mathcal{E}}(\mu,\nu)$ because if $(m,v,\rho)\in \mathcal{E}(\mu)$ then for
$\eta(dt,dx, du)=\delta_{v(t,x)}(du)m(t,x)dtdx$, we have $(\eta,\rho)\in \tilde{\mathcal{E}}(\mu)$ with the same cost and target distribution. The following proposition shows that from a probability distribution, drift, and stopping time $(\mathbb{P},\beta,\alpha) \in \mathcal{A}(\mu)$, we can construct an admissible pair $(\eta,\rho)\in\tilde{\mathcal{E}}(\mu)$.

\begin{proposition} \label{prop:Eulerian_embedding}
Given $(\mathbb{P},\beta,\alpha) \in \mathcal{A}(\mu)$, there is a pair $(\eta,\rho)\in \tilde{\mathcal{E}}(\mu)$ such that
$$\int_{\R^+}\int_{\R^d}\int_{U} L(t,x,u) \eta(dt,dx, \mathrm{d}u)=\mathbb{E}^{\alpha\ltimes\mathbb{P}}\Big[\int_{0}^\tau L(t,X_t,\beta_t)dt\Big].
$$
In particular, we have
$$ \mathcal{P}^{\tilde{\mathcal{E}}}_L(\mu,\nu) \leq \mathcal{P}_L(\mu,\nu).
$$
\end{proposition}
\begin{proof}
Given $(\mathbb{P},\beta,\alpha) \in \mathcal{A}(\mu)$, we find the pair $(\eta,\rho)$ from Riesz representation given by the formula,
$$
\int_{\mathbb{R}^+}\int_{\mathbb{R}^d} \int_{U} \phi(t,x,u)\eta(dt,dx,du)=\mathbb{E}^{\alpha\ltimes\mathbb{P}}\Big[\int_0^\tau \phi(t,X_t,\beta_t)dt\Big],\,\,\mbox{for all}\,\,\,\phi \in C_b(\mathbb{R}^+ \times \mathbb{R}^d \times U),$$
and
$$\int_{\mathbb{R}^+}\int_{\mathbb{R}^d} h(\tau,y)\rho(d\tau,dy)=\mathbb{E}^{\alpha\ltimes\mathbb{P}}\big[h(\tau,X_\tau)\big],\,\,\mbox{for all}\,\,\,h\in C_b(\mathbb{R}^+ \times \mathbb{R}^d).$$
Then, one can check easily that the pair $(\eta,\rho) \in \tilde{\mathcal{E}}(\mu)$. In fact, if  $\eta_t$ is the disintegration of $\eta$ with respect to the Lebesgue measure on $\R^+$ then, 
 for every
 smooth $\phi$ with compact support in $\mathbb{R}^+ \times \mathbb{R}^d$, we have using Ito's formula the following:
\begin{align*}
    &\ \qquad \,\, \int_{\mathbb{R}^+} \int_{\mathbb{R}^d} \int_{U} \Big[\partial_t \phi(t,x) + \frac{1}{2} \Delta \phi(t,x) + u \cdot \nabla \phi(t,x)\Big]\eta_t(dx,du)dt\\
    &\ \qquad\qquad=\,\, \mathbb{E}^{\alpha\ltimes\mathbb{P}}\Big[\int_{0}^\tau \Big(\partial_t \phi(t,X_t)  +\beta_t\cdot \nabla \phi(t,X_t) + \frac{1}{2}\Delta \phi(t,X_t)\Big) dt\Big]\\
    &\ \qquad\qquad=\,\, \mathbb{E}^{\alpha\ltimes \mathbb{P}}\Big[\phi(\tau,X_\tau) - \phi(0,X_0)\Big]=\int_{\R^d}\int_{\R^+} \phi(\tau,y)\rho(d\tau,dy)-\int_{\R^d} \phi(0,x)\mu(dx).
\end{align*}
Clearly, we have that if $X_\tau\sim_{\alpha\ltimes \mathbb{P}}\nu$\, then
\,$\int_{\mathbb{R}^+} \rho(d\tau,\cdot)= \nu$. Moreover, we have the following:
\begin{equation*}
\int_{\mathbb{R}^+} \int_{\mathbb{R}^d} \int_{U} L(t,x,u) \eta(dt,dx,du)=\mathbb{E}^{\alpha\ltimes\mathbb{P}} \Big[\int_0^\tau L(t,X_t,\beta_t)dt\Big],
\end{equation*}
which completes the proof that $ \mathcal{P}^{\tilde{\mathcal{E}}}_L(\mu,\nu) \leq \mathcal{P}_L(\mu,\nu)$. $\qedhere$
\end{proof} 

On the other hand, we have the following duality for the convex Eulerian problem (\ref{eqn:Eulerian_primal_convex}).
 \begin{theorem}\label{thm:Eulerian_duality}
 The following equalities hold:
 $$
 D_L(\mu,\nu)= \mathcal{P}_L^{\tilde{\mathcal{E}}}(\mu,\nu)=\mathcal{P}_L(\mu,\nu).
 $$
 \end{theorem}
 \begin{proof}
 Take an admissible pair $(\eta,\rho) \in \tilde{\mathcal{E}}(\mu)$ with $\int_{\R^+}\rho(d\tau,\cdot)=\nu$ and let $(J,\psi)$ satisfy \eqref{HJBESec.4}.
   Then, we have
\begin{align*}
 &\ \qquad\qquad \int_{\mathbb{R}^d} \psi(y)\,\nu(dy) - \int_{\mathbb{R}^d} J(0,x)\,\mu(dx)\\
  &\ \qquad \qquad \qquad \leq \int_{\mathbb{R^+}}\int_{\mathbb{R}^d} J(t,x)\,\rho(dt,dx) - \int_{\mathbb{R}^d} J(0,x)\,\mu(dx)\\
&\ \qquad\qquad \qquad= \int_{\mathbb{R}^+} \int_{\mathbb{R}^d} \int_{U} \Big[\partial_t J(t,x) + \frac{1}{2} \Delta J(t,x) +u \cdot \nabla J(t,x)\Big]\eta_t(dx,du)dt \\
&\ \qquad\qquad \qquad\leq  \int_{\mathbb{R}^+} \int_{\mathbb{R}^d}\int_{U} L(t,x,u)\eta(dt,dx,du).
\end{align*}
This shows that $\mathcal{D}_L(\mu,\nu)\leq \mathcal{P}_L^{\tilde{\mathcal{E}}}(\mu,\nu)$ in view of Proposition \ref{prop:dual_equivalence}. We have shown that  $\mathcal{P}_L^{\tilde{\mathcal{E}}}(\mu,\nu)\leq \mathcal{P}_L(\mu,\nu)$ in Proposition \ref{prop:Eulerian_embedding}, and that $\mathcal{P}_L(\mu,\nu)=\mathcal{D}_L(\mu,\nu)$ in Theorem \ref{thm:stochastic_duality}, which completes the chain of equalities. $\qedhere$
\end{proof}

\subsection{Regularity
}
We will now partly complete the equivalence (between the strong and the convex Eulerian formulations) by addressing the strong Eulerian formulation in the case where the drift is bounded. We first need a result on the truncation in time and space of pairs $(\eta,\rho)\in \mathcal{E}(\mu)$.  We make use of the truncation in time and space of Lemma \ref{lem:truncation}.
 
\begin{theorem}\label{thm:Eulerian_regularity}
Suppose that 
$\mu= \hat{\mu}\,dx$ for $\hat{\mu} \in L^2(\R^d)$.  Then, for any 
$(\eta,\rho)\in {\tilde{\mathcal{E}}}(\mu)$ with compact support in $\R^+\times \R^d\times U$ ($0\leq t < T$, $|x|< R$, $|u|< \overline{\mathfrak{u}}$), there is a $(m,v,\rho)\in \mathcal{E}(\mu)$ such that $\eta_t(dx,\R^d)=m(t,x)dx$, and $(m,v,\rho)$ has a cost less than or equal to that of $(\eta,\rho)$. Moreover, we have
the following uniform estimate 
\begin{align}\label{eqn:energy_estimate}
\int_{0}^T\int_{\R^d} |\nabla m(t,x)|^2dx\, dt\leq 2\|\mu\|_{L^2(\R^d)}^2+CT,
\end{align}
where the constant $C$ depends only on the bound of the drift \,$\overline{\mathfrak{u}}$ and the dimension $d$.
In particular,
$$
    f\mapsto \int_{\R^d}\int_0^T f(x)\rho(d\tau,dx)
$$
is a continuous linear functional of $H^1(\R^d)$. Finally, when the drift is bounded, we have
$$
    \mathcal{P}_L^{\tilde{\mathcal{E}}}(\mu,\nu)=\mathcal{P}_L^{\mathcal{E}}(\mu,\nu).
$$
\end{theorem}
\begin{proof}
First, we use convolution to approximate the pair 
 $(\eta,\rho)$ by smooth densities $(\eta^\epsilon,\rho^\epsilon)$ and the measure $\mu$ by $\mu^\varepsilon$ with $(\eta^\epsilon,\rho^\epsilon)\in \tilde{\mathcal{E}}(\mu^\epsilon)$. We define 
    $$    m^\epsilon(t,x)= \int_{\mathbb{R}^d} \eta^\epsilon(t,x,u)du
    $$
    and
    $$    v^\epsilon(t,x) = \frac{\int_{\mathbb{R}^d} u\, \eta^\epsilon(t,x,u)du}{m^\epsilon(t,x)},
    $$
    with $v^\epsilon(t,x)=0$ if $m^\epsilon(t,x)=0$.
    We note that $(m^\epsilon,v^\epsilon,\rho^\epsilon)\in \mathcal{E}(\mu^\epsilon)$.
    By Jensen's inequality, the cost of $(m^\epsilon,v^\epsilon,\rho^\epsilon)$ is less or equal to that of $(\eta,\rho)$ within a factor of $\epsilon$. Then, using $m^\epsilon$ as a test function in (\ref{eqn:strong_Eulerian_evolution}), we obtain
    \begin{align*}
    &\ \int_{\R^d}\int_{0}^Tm^\epsilon (\tau,y)\rho^\epsilon(d\tau,dy)-\int_{\R^d} \big(\hat{\mu}^\epsilon(x)\big)^2dx\\
    & \qquad =\int_{0}^T\int_{\R^d}\Big[\frac{1}{2}\partial_t \big(m^\epsilon(t,x)^2\big)-\frac{1}{2}\big|\nabla m^\epsilon(t,x)\big|^2+v^\epsilon(t,x)\cdot \nabla m^\epsilon (t,x)\, m^\epsilon(t,x)\Big]dx\, dt.
    \end{align*}
     Then,
    \begin{align*}
    &\ \frac{1}{2}\int_{0}^T\int_{\R^d}|\nabla m^\epsilon(t,x)|^2dx\, dt \leq \frac{1}{2} \|\hat{\mu}^\epsilon\|_{L^2(\R^d)}^2+
\overline{\mathfrak{u}}\int_{0}^T\int_{\R^d} |\nabla m^\epsilon(t,x)|  m^\epsilon(t,x)dx\, dt.
    \end{align*}
    Yet, we have
  $$  \int_{0}^T\int_{\R^d} |\nabla m^\epsilon(t,x)|  m^\epsilon(t,x)dx\, dt \leq \frac{1}{2}\int_{0}^T\int_{\R^d} \bigg(\delta|\nabla m^\epsilon(t,x)|^2 + \frac{1}{\delta} {m^\epsilon(t,x)}^2\bigg) dx\, dt.$$
 For the last term 
 using Gagliardo-Nirenberg-Sobolev embedding 
 and an interpolation of the $L^1$ and $L^{2^*}$ norms, we have
\begin{align*}
    \|m^\epsilon_t\|_{L^2( \R^d)}\leq \ \delta \|\nabla m^\epsilon_t\|_{L^2(\R^d)}
    + C(\delta)\|m^\epsilon_t\|_{L^1(\R^d)},
\end{align*}
which implies that
  $$  \int_{0}^T\int_{\R^d} {m^\epsilon(t,x)}^2dx\, dt \leq 2{\delta}^2 \int_{0}^T\int_{\R^d} |\nabla m^\epsilon(t,x)|^2 dx\, dt + 2 C(\delta)^2 \int_{0}^T\bigg(\int_{\R^d} m^\epsilon(t,x) dx\bigg)^2\, dt.$$
Yet, $\|m^\epsilon(t,\cdot)\|_{L^1(\R^d)} \leq 1 + \epsilon$, for every $t$. Then, choosing $\delta>0$ small enough,
 we get (\ref{eqn:energy_estimate}) for $m^\epsilon$.
    In particular, the uniform estimates imply that $m^\epsilon$ converge weakly to $m$ in $L_{loc}^2(\R^+,H^1(\R^d))$ such that
    $$
        m(t,\cdot) =  \frac{\int_{\R^d}\eta_t(\cdot,du)}{dx}.
    $$
    We then define the field $v(t,x)$ by the vector-valued Radon-Nikodym derivative
    $$
        v(t,\cdot) = \frac{\int_{\R^d}u\, \eta_t(\cdot,du)}{m_t\, dx}.
    $$
     It is then straightforward to see $(m,v,\rho)\in \mathcal{E}(\mu)$ which by Jensen's inequality, has lesser or equal cost than $(\eta,\rho)$.
     
     For any $f\in H^1(\R^d)$, we have
$$
	\int_{\R^d}\int_{\R^+}f(x)\rho(d\tau,dx)$$
	$$=\int_{\R^d} f(x)\hat{\mu}(x)\,\mathrm{d}x + \int_{0}^T\int_{\mathbb{R}^d}\Big[-\frac{1}{2}\nabla f(x)\cdot \nabla m(t,x)+v(t,x)\cdot \nabla f (x) \, m(t,x)\Big]dx\, dt,
$$
and the similar estimates imply that $f\mapsto \int_{\R^+}\int_{\R^d} f(x)\rho(d\tau,dx)$ is a linear functional of $H^1(\R^d)$.

Finally, to conclude equivalence of the convex and strong formulations, we use Proposition \ref{prop:Eulerian_embedding} and Lemma \ref{lem:truncation} to construct $(\eta^{T,R},\rho^{T,R})$ with compact support in time and space for any $(\mathbb{P},\beta,\alpha)\in \mathcal{A}(\mu)$, corresponding to the truncation of Lemma \ref{lem:truncation} with $T,\,R \in \R^+$.
We then have corresponding $(m^{T,R},v^{T,R},\rho^{T,R})\in \mathcal{E}(\mu)$.  When $T_1<T_2$ and $R_1<R_2$, we have $(\eta^{T_2,R_2},\rho^{T_2,R_2})=(\eta^{T_1,R_1},\rho^{T_1,R_1})$ on $[0,T_1] \times B(0,R_1)$, and thus there exists a density $(m,v,\rho)\in \mathcal{E}(\mu,\nu)$ taking the limit as $T,\,R\rightarrow \infty$, with the same cost and end distribution as $(\mathbb{P},\beta,\alpha)$.
Thus, we have shown that
$$
	\mathcal{P}_L^\mathcal{E}(\mu,\nu)=\mathcal{P}_L(\mu,\nu),
$$
and applying
the result of Theorem \ref{thm:Eulerian_duality} completes the proof of equivalence.
\end{proof}

    We will also need the following moment bound.
\begin{proposition}\label{prop:moment_bound}
    Suppose that \,$c(|u|^p +|x|^q+1) \leq L(t,x,u)  
    $  
    or \,$c(|x|+1)\leq L(t,x,u)$ and $u$ is uniformly bounded, for $c>0$. 
    We assume that $\mu$ satisfies $\int_{\R^d} |x|^2 \mu(dx)<+\infty$.  Then, for any $(\eta,\rho)\in \tilde{\mathcal{E}}(\mu)$ with finite cost, we have
    $$
        \int_{\R^d}\int_{\R^+} |y|^2\rho(d\tau,dy)\leq \int_{\R^d}|x|^2\mu(dx)+C\int_{\R^+}\int_{\R^d}\int_{U}L(t,x,u)\eta_t(dx,du)dt.
    $$
\end{proposition}
\begin{proof}
To prove that, we simply apply (\ref{eqn:convex_Eulerian_evolution}) with the test function $w(y)=|y|^2$:
\begin{align*}
 \int_{\R^d}\int_{\R^+} |y|^2\rho(d\tau,dy)-\int_{\R^d}|x|^2\mu(dx) &=  \int_{\R^+}\int_{\R^d}\int_{U}\Big[d+u\cdot 2x\Big]\eta_t(dx,du)dt\\
	&\leq \int_{\R^+}\int_{\R^d}\int_{U}\Big[ d+H\big(t,x,2x\big)+L(t,x,u)\Big]\eta_t(dx,du)dt\\
	&\leq
	 C\int_{\R^+}\int_{\R^d}\int_{\R^d}\Big[L(t,x,u)\Big]\eta_t(dx,du)dt,
\end{align*}
	where we have used that
	$$
	H\big(t,x,2x\big)\leq C |x|^{q}\leq C\, L(t,x,u). 
	$$
	
\noindent In the case that $c|x|\leq L(t,x,u)$ and $u$ is uniformly bounded, the proof is simpler, using  that 
$$
u\cdot x\leq \frac{\overline{\mathfrak{u}}}{c}L(t,x,u),
$$
and the result follows. $\qedhere$
\end{proof}

\section{Dual attainment}\label{sec:attainment-large-p}
In this section, we prove dual attainment in the cases where either the Lagrangian $L \approx |u|^p$ with $1 <p<2$ or the drift is uniformly bounded ($|u|\leq \overline{\mathfrak{u}}$). This relies on a normalization that makes $\psi$ as a supersolution to an HJB equation. First, we define
$$    \bar{H}(x,z)=\inf_{t \in \mathbb{R}^+}H(t,x,z),\,\,\,\mbox{for all}\,\,\,(x,z) \in \mathbb{R}^d \times \mathbb{R}^d,
$$
and
$$\bar{L}(x,u)=\sup_{t \in \mathbb{R}^+} L(t,x,u),\,\,\,\mbox{for all}\,\,\,(x,u) \in \mathbb{R}^d \times U.$$
We suppose, strengthening assumption (\ref{eqn:L-assumption-1}), that there are constants $c,C>0$ such that
$L$ satisfies
\begin{align}\label{eqn:Hamiltonian_coercive}
    c\big(|u|^p +|x|^q + 1)\leq  L(t,x,u)\leq C(|u|^p+|x|^q+1)
\end{align}
for all $(t,x,u) \in \R^+\times \mathbb{R}^d \times U$,
or equivalently, there are constants $\lambda,\Lambda,c,C>0$ such that $H$ satisfies
$$
       \lambda |z|^q-C\big(|x|^q+1\big)\leq H(t,x,z)\leq \Lambda |z|^q -c\big(|x|^q+1),\,\,\,\mbox{for all}\,\,\,(t,x,z) \in \mathbb{R}^+ \times \mathbb{R}^d \times \mathbb{R}^d.
$$
In the case where the drift is bounded, this becomes
$$
    c(|x|+1)\leq L(t,x,u)\leq C(|x|+1),\ {\rm for\,\, all\ }(t,x,u) \in \mathbb{R}^+ \times \mathbb{R}^d \times U,
$$
and
$$
       \lambda |z|-C\big(|x|+1\big)\leq H(t,x,z)\leq \Lambda |z| -c\big(|x|+1),\,\,\,\mbox{for all}\,\,\,(t,x,z) \in \mathbb{R}^+ \times \mathbb{R}^d \times \mathbb{R}^d.
$$

\begin{proposition}\label{prop:dual_normalization}
For $\psi\in C_b(\R^d)$,
 we define
 $$
    \bar\psi(x):=\sup_{(\mathbb{P},\beta,\alpha)\in \mathcal{A}(\delta_x)}\mathbb{E}^{\alpha\ltimes \mathbb{P}}\Big[\psi(X_\tau)-\int_0^\tau \bar{L}(X_t,\beta_t)dt\Big].
$$
 Then, $\bar{\psi}$ is lower semi-continuous and bounded with $J_{\bar \psi}=J_\psi$ and thus greater or equal dual value. Furthermore, $\bar \psi$ satisfies, in the viscosity sense,
    \begin{align}\label{eqn:psi_supersolution}
        \frac{1}{2}\Delta \bar \psi(x)+\overline{H}\big(x,\nabla \bar\psi(x)\big)\leq 0.
    \end{align}
\end{proposition}  
\begin{proof}
From the definition of $\bar{\psi}$, we have obviously $\bar{\psi} \geq \psi$, and we get as in Lemma \ref{lem:measurability} that $\bar\psi$ is lower semi-continuous and bounded.  The standard viscosity solution theory implies that $\bar\psi$ is a viscosity supersolution of (\ref{eqn:psi_supersolution}).

From the definition of $J_\psi$, we get that $J_{\bar{\psi}} \geq J_{\psi}$. Let us prove the reverse inequality, that is $J_{\bar\psi} \leq J_{\psi}$, so that we get $J_{\bar{\psi}}=J_\psi$. First, we note that $\bar\psi(x)\leq J_\psi(t,x)$, for every $(t,x) \in \mathbb{R}^+ \times \mathbb{R}^d$, which
follows from the definitions after noting that $\bar{L}\geq L$.

 Now, suppose that $(\mathbb{P},\beta,\alpha)\in \mathcal{A}_t(\delta_x)$ is within $\epsilon$ of optimality for $J_{ \bar \psi}$, then we have the following:
\begin{align*}
    J_{\bar\psi}(t,x)-\epsilon\leq&\ \mathbb{E}^{\alpha\ltimes \mathbb{P}}\Big[\bar\psi(X_\tau)-\int_t^\tau L(s,X_s,\beta_s)ds\Big]\\
    \leq&\ \mathbb{E}^{\alpha\ltimes \mathbb{P}}\Big[J_\psi(\tau,X_\tau)-\int_t^\tau L(s,X_s,\beta_s)ds\Big]\leq J_\psi(t,x),
\end{align*}
where the last inequality is a result of dynamic programming principle \cite[Theorem 6, Ch.\,3]{krylov1980controlled} for $J_\psi$. Taking $\epsilon$ to zero proves the desired inequality. $\qedhere$ 
\end{proof}

 The following proposition proves a quadratic lower bound on supersolutions to (\ref{eqn:psi_supersolution})
 as well as an H\"{o}lder continuity in the case that $L \approx |u|^p$ with $1<p<2$.

\begin{proposition}\label{prop:Holder_continuity}
  We assume \eqref{eqn:Hamiltonian_coercive} holds. Suppose $\psi$ is bounded, lower semi-continuous and satisfies (\ref{eqn:psi_supersolution}).  Suppose $d>1$ and $1<p<2$. 
  Fix \,$0<\delta \leq 2-p <1$. Then, for each $x_0 \in \mathbb{R}^d$, there are two constants $B$ and $E$ (depending only on $\delta,\,p,\,d,\,\lambda,\,C$ and $|x_0|)$ 
  such that
        \begin{align}\label{eqn:global_inequality}
            \psi(x_0)-\psi(x_1)\leq B|x_1-x_0|^\delta +E |x_1-x_0|^2, \,\, \hbox{ for all\, $x_1 \in \R^d$.}
        \end{align}
        In particular,
        $\psi$ is uniformly $\delta$-H\"{o}lder continuous on compact sets and, under the assumption that $\psi(0)=0$, $\psi$ is uniformly globally bounded from below by a quadratic function.
         
        In the case  $d=1$, the result holds with $\delta=1$ for all $p>1$.
    \end{proposition}
    \begin{proof}
   We will prove that, for each $x_0$, the  function
        $$
            w(x)=A-B|x-x_0|^\delta-E|x-x_0|^2
$$
 with appropriate constants $A$, $B$
 and $E$,  will touch $\psi$ from below at $x_0$.
        By computing $\nabla w$ at $x\not=x_0$, we see that
    $$     \lambda|\nabla w(x)|^q \geq b_1 |x-x_0|^{q(\delta-1)}+b_2 |x-x_0|^{q}
        $$
                where  $$b_1=\delta^q \lambda B^q \,\hbox{ and }\, b_2= \lambda 2^q E^q.$$ 
                Notice that as $\psi$ is l.s.c. and bounded,  we can by adjusting the constant $A$, let the function $w$ touch $\psi$ from below at some point.
        Suppose $w$ touches $\psi$ from below at $x_1\not=x_0$.  Then, from the equation \eqref{eqn:psi_supersolution} and the assumption \eqref{eqn:Hamiltonian_coercive}, we have that
\begin{align*}
    C\big(|x_1|^q +1 \big)\geq&\ \frac{1}{2}\Delta w(x_1)+\lambda |\nabla w(x_1)|^q\\
            \geq&\ b_1 |x_1-x_0|^{q(\delta-1)}+b_2 |x_1-x_0|^{q} -\frac{\delta B }{2}(d+ \delta-2)|x_1-x_0|^{\delta-2} -dE.
        \end{align*}
        To draw a contradiction,
        we will select two constants $B$ and $E$ such that the following holds
$$b_1 |x_1-x_0|^{q(\delta-1)}+b_2 |x_1-x_0|^{q} $$
            $$> 3\max\left\{\frac{\delta B}{2}(d+\delta-2)|x_1-x_0|^{\delta-2}, \ dE+C, \ C|x_1|^q\right\}.$$
            Now, it is clear that the inequality with the first and the second item in the maximum is satisfied for large enough $B$ and $E$ such that 
        \begin{align*}
            & \min\big\{b_1,b_2 \big\} > 3\max \Big\{ \frac{\delta B}{2}(d+\delta-2),d E+C \Big\}.
            \end{align*}
            For the third item, note that $$|x_1|^q\leq 2^q\max\big\{|x_1-x_0|,|x_0|\big\}^q.$$ Thus, we may choose $B$ and $E$ large enough, depending on $|x_0|$, such that $b_1> 3C(1+|x_0|)^q$ and $b_2 >3C 2^q \max\{1,|x_0|^q\}$, and  these choices of $B$ and $E$ yield a contradiction.  Thus, we must have that $w$ touches $\psi$ from below at $x_1=x_0$, and it follows that (\ref{eqn:global_inequality}) holds for all $x_0$ and $x_1$.  We note that in particular, fixing $x_0$, this establishes a global quadratic lower bound on $\psi$, which becomes uniform (in $\psi$) as soon as  $\psi(0)=0$.
       
        The proof for $d=1$ is simpler with $\delta=1$ since the Laplacian of the second term  of $w$ vanishes. $\qedhere$
    \end{proof}

    \begin{theorem}\label{thm:dual_attainment_small_p}
        Suppose that $d\geq 2$, $1<p<2$, $\int_{\R^d} |x|^2 \mu(dx)<+\infty$,
        $\nu$ has a compact support, and $L$ satisfies \eqref{eqn:Hamiltonian_coercive}.
        Then, the dual problem $\mathcal{D}_L(\mu , \nu)$  is attained at $\psi\in C_{loc}^{\delta}(\R^d)$, $0< \delta < 2-p$,  
        and $\psi$ is globally bounded below by a quadratic function.
       
        Furthermore, in this case there exists $(\mathbb{P},\beta,\alpha)\in \mathcal{A}(\mu,\nu)$ with finite cost.
    \end{theorem}
    \begin{proof}
        We take a maximizing sequence $\{\psi^i\}$. We assume that, for each $i \in \mathbb{N}$, $\psi^i$ is l.s.c., bounded, $\psi^i(0)=0$ and $\psi^i$ satisfies (\ref{eqn:psi_supersolution}) (this is possible thanks to Proposition \ref{prop:dual_normalization}), and we apply Proposition~\ref{prop:Holder_continuity}. Then,  by Arzela-Ascoli, $\{\psi^i\}$ converges uniformly on compact sets to $\psi$ with $\psi(0)=0$.  Furthermore, these $\psi^i$ are uniformly bounded below by a quadratic function.     With the target measure $\nu$ compactly supported, for such a  limit function $\psi$ to have the maximal dual value,     it is enough to show that
\begin{equation}\label{semi-continuity Holder}
\liminf_i\int_{\R^d}J_{\psi^i}(0,x)\mu(dx) \geq \int_{\R^d}J_\psi(0,x)\mu(dx).
\end{equation}
But again, we have
\begin{align*}
  \int_{\R^d}J_{\psi^i}(0,x)\mu(dx)  =&\ \sup_{(\eta,\rho)\in \tilde{\mathcal{E}}(\mu)}\Big[ \int_{\R^+}\int_{\R^d} {\psi}^i (y)\rho(dt,dy)-\int_{\R^+}\int_{\R^d}\int_{U}L(t,x,u)\eta(dt,dx,du)\Big].
\end{align*}
For each $(\eta,\rho)\in \tilde{\mathcal{E}}(\mu)$ with finite cost,  Proposition~\ref{prop:moment_bound} implies that such $\rho$ has finite second moment. Thus, the uniform quadratic lower bound on $\psi^i$ implies that
$$ \liminf_i\int_{\R^+}\int_{\R^d} \psi^i (y) \rho(dt,dy) \geq \int_{\R^+}\int_{\R^d} \psi (y) \rho(dt,dy).$$
This yields \eqref{semi-continuity Holder}.
%
        Finally, the existence of $(\mathbb{P},\beta,\alpha)\in \mathcal{A}(\mu,\nu)$ with finite cost follows from duality (Theorem~\ref{thm:stochastic_duality}). $\qedhere$
    \end{proof}
   
    In 1D the result holds for all values of $p>1$.
    \begin{theorem}\label{thm:dual_attaiment_1D}
        Let $d=1$.  Assume that 
      $\int_{\R^d} |x|^2 \mu(dx)<+\infty$, $\nu$ has a compact support and $L$ satisfies \eqref{eqn:Hamiltonian_coercive}. Then, the dual problem is attained at $\psi$, which is locally Lipschitz continuous and bounded below quadratically.
    \end{theorem}
    \begin{proof}
        The proof follows as in Theorem \ref{thm:dual_attainment_small_p} using the Lipschitz estimates in 1D from Proposition \ref{prop:Holder_continuity}. $\qedhere$
    \end{proof}

In order to better handle the behavior as  $|x|\rightarrow +\infty$ in the case that $p$ is large (for instance, assume that $p=+\infty$ which means that the drift is bounded; $|u| \leq \overline{\mathfrak{u}}$), we will introduce a weighting measure on $\R^d$.  We consider now a smooth convex function $V$ such that $\int_{\R^d} e^{-V(x)}dx=1$ and $\frac{1}{2}|\nabla V|\leq \lambda$, for instance $V(x)=\gamma|x|+m$ (we assume that $\lambda > \frac{\gamma}{2}$)  is easily seen to fit all the criteria we will require.    We define the norm 
$$
    \|f\|_{{L^q_V}(\mathbb{R}^d)}= \bigg(\int_{\R^d} |f(x)|^q e^{-V(x)}dx \bigg)^{\frac{1}{q}}.$$
Set
    $$W^{1,q}_V(\mathbb{R}^d):=\Big\{f;\  \|f\|_{L^q_V(\mathbb{R}^d)} +     \|\nabla f\|_{L^q_V(\mathbb{R}^d)} < + \infty \Big\}.$$
    We denote
    \begin{align*}
  (f)_V=\int_{\R^d}f(x)e^{-V(x)}dx.
\end{align*}
We note that if \,$V$ is a radially increasing function, then we have, for any $q \geq 1$, the following Poincar\'{e} inequality (see, for example, \cite{Elcrat,Dyda}):
$$
   \|f-(f)_V\|_{L^q_V(\R^d)}\leq C\|\nabla f\|_{L^q_V(\R^d)}.
$$

\noindent Then, we define
\begin{align}\label{eqn:B-V}
\mathcal{B}:= \left\{f \in L^1_V(\R^d),\,
f\ \mbox{is l.s.c and quadratically bounded from below}
\right\}.
\end{align}
\noindent 
We will say that $f\in H^1_V(\R^d):=W^{1,2}_V(\R^d)$ satisfies (\ref{eqn:psi_supersolution}) in a weak sense if for all compactly supported $h\in H^1_V(\R^d)$ with $h \geq 0$, we have
$$ \int_{\R^d}\Big[-\frac{1}{2}\nabla f(x) \cdot\nabla h(x)+\bar{H}\big(x,\nabla f(x)\big)h(x) \Big]dx \leq 0.
$$

\noindent \hspace{0.1cm} Then, we have the following:
    \begin{proposition} \label{prop:weak_solution}
    Suppose that $\bar{H}$ is uniformly continuous in $x$ (uniformly w.r.t. $z$) and satisfies {\normalfont (\ref{eqn:Hamiltonian_coercive})}.  
    Let $\psi$ be a supersolution to \eqref{eqn:psi_supersolution}, lower semi-continuous and bounded from below with $(\psi)_V=0$, then, for any $M\geq  0$, the truncation of $\psi$ satisfies $\psi\wedge M\in H^1_V(\R^d)$ and solves {\normalfont (\ref{eqn:psi_supersolution})} in a weak sense.
    More precisely, there exists a uniform constant $\overline{C}$ (which does not depend on $\psi$)  such that
    \begin{align}\label{eqn:Lq_bound}
        \|\psi\|_{W^{1,1}_V(\R^d)}\leq \overline{C},
    \end{align}
    \begin{align}\label{eqn:quadratic_bound}
        \psi \geq -\overline{C}(1 + |x|^2),
    \end{align}
    and, for every \,$M \geq0$, 
    there is a constant $C(M)$ (again uniform in $\psi$) such that the following holds:
    \begin{align}\label{eqn:H1_bound}
        \|\psi \wedge M\|_{H^1_V(\R^d)}\leq C(M).
    \end{align}
On the other hand, if 
the truncation
\,$\psi\wedge M\in H^1_V(\R^d)$  solves {\normalfont (\ref{eqn:psi_supersolution})} weakly for all $M\geq0$,
then $\psi$ is l.s.c.
    \end{proposition}
   
    \begin{proof}
        First, we fix $M\geq 0$ and show there is a sequence $(\psi^M_\varepsilon)_\varepsilon$ such that $\psi^M_\varepsilon$ is $1/\varepsilon-$Lipschitz and semi-concave, $\psi^M_\varepsilon \rightharpoonup \psi\wedge M$ in $H^1_V(\R^d)$ and $\psi_\varepsilon^M$ satisfies (in a strong sense)
        \begin{equation}\label{Approx viscosity solution}
    \frac{1}{2}\Delta  \psi_\varepsilon^M(x)+\overline{H}\big(x,\nabla \psi_\varepsilon^M(x)\big)\leq C(M,\varepsilon).
    \end{equation}
        Set $\psi_\varepsilon^M(x)=\inf_{z \in \mathbb{R}^d}\{\psi(z)\wedge M + \frac{1}{2 \varepsilon^2} |x-z|^2\}$, for all $x \in \mathbb{R}^d$. We let $z_x$ (depending also on $\varepsilon,M$) be such that $\psi_\varepsilon^M(x)=\psi (z_x) \wedge M + \frac{1}{2 \varepsilon^2} |x-z_x|^2$, which is well defined for sufficiently small $\varepsilon$ by lower semi-continuity and the lower bound on $\psi$. 
        Fix $\varepsilon>0$ and $x\in \R^d$ and let $\phi$ be a smooth function such that $\phi \leq \psi_\varepsilon^M$ with $\phi(x)=\psi_\varepsilon^M(x)$. We define $w(y):=\phi(y+x-z_x)-\frac{1}{2 \varepsilon^2} |x-z_x|^2.$ We have $w(z_x)=\psi(z_x)\wedge M$ and $w \leq \psi\wedge M$. Hence, by (\ref{eqn:psi_supersolution}) and that $\overline{H}(x,0)\leq 0$, $$\frac{1}{2}\Delta  w(z_x)+\overline{H}\big(z_x,\nabla w(z_x)\big)\leq 0.
    $$
    Yet, $\nabla w(z_x)=\nabla \phi(x)$ and $\Delta w(z_x)=\Delta \phi(x)$. Using also the fact that $\bar{H}$ is uniformly continuous, we get
    $$\frac{1}{2}\Delta  \phi(x)+\overline{H}\big(x,\nabla \phi(x)\big)\leq C(|x-z_x|).
    $$
    Moreover, we have
    $$\psi(z_x)\wedge M + \frac{1}{2 \varepsilon^2} |x-z_x|^2 \leq \psi(x)\wedge M. $$
    Hence, $|x-z_x|\leq C(M,\varepsilon)$, and it follows that $\psi_\varepsilon^M$ is a supersolution in the viscosity sense with error $C(M,\varepsilon)$ where $C(M,\varepsilon)\rightarrow 0$ as $\varepsilon\rightarrow 0$, but also 
    in the sense of distributions thanks to the semi-concavity of $\psi_\varepsilon^M$.
    Using $e^{-V}$ as a test function, we obtain
    \begin{align*}
        C(M,\varepsilon)\geq&\ \int_{\R^d}\Big[ \frac{1}{2}\Delta \psi_\varepsilon^M(x)+\overline{H}\big(x,\nabla \psi_\varepsilon^M(x)\big)\Big]e^{-V(x)}dx\\
\geq&\ -\frac{1}{2}\int_{\R^d}|\nabla \psi_\varepsilon^M(x)|\, |\nabla V(x)|e^{-V(x)}dx+\lambda\|\nabla \psi_\varepsilon^M\|_{L^1_V}-C.
\end{align*}  
Yet, $\gamma<2\lambda$.
Then, we infer that
$$
\|\nabla \psi_\varepsilon^M\|_{L^1_V(\R^d)}\leq C,
$$
which shows the estimate (\ref{eqn:Lq_bound}) after taking $\varepsilon\rightarrow 0$. 
For the uniform quadratic lower bound: let us consider as in Proposition \ref{prop:Holder_continuity} the function
$w(y):= A-\frac{E}{2}|y|^2$ that touches $\psi_\varepsilon^0$ from below at $x$.  We then have that
$$
    -\frac{d}{2}E+ \lambda E|x|-C(|x| +1)\leq C(\varepsilon),
$$
which implies that, for sufficiently large $E$, $x$ is in a ball of radius $C_1=C_1(d,\,\lambda,\,C,\,E)$ (we note that the constant $E$ can be taken independent of $x$). Then,
$$
  \psi_\varepsilon^0(y)\geq \inf_{|x|\leq C_1}\psi_\varepsilon^0(x)+\frac{E}{2}|x|^2-\frac{E}{2}|y|^2. 
$$
To bound $\psi_\varepsilon^0$ near the origin
we will consider a construction that will be useful for the remainder of the proof.
For any $x$, we let $Q_{x,R}$ denote the uniform distribution on the ball of radius $R$ centered at $x$. Fix $0< R_1<R_2$ and assume that the drift equals zero, then we can find a pair $(\eta,\rho)=(m,0,\rho)\in {\mathcal{E}}(Q_{x,R_1},Q_{x,R_2})$ such that $m$ has support in $[0,T] \times Q_{x,R_2}$. 
From (\ref{eqn:convex_Eulerian_evolution}), we have
\begin{align*}
    \int_{\R^d}\frac{|y-x|^2}{d}Q_{x,R_2}(dy)=&\ \int_{\R^+}\int_{\R^d}\int_{\R^d}\eta_t(dy,du)dt +\int_{\R^d}\frac{|y-x|^2}{d}Q_{x,R_1}(dy),
\end{align*}
which implies that
$$
   \int_{\R^+}\int_{\R^d} m(t,y)dy\, dt
    = C (R_2^2-R_1^2).
$$
Define
$$
    \psi_\varepsilon^{M,R}(x):=\int_{\R^d}\psi_\varepsilon^M(y)Q_{x,R}(dy).
$$
As \,$m(t,\cdot)\in H^1(\R^d)$ (see Theorem \ref{thm:Eulerian_regularity}), for each $t \in [0,T]$, with compact support, then again by (\ref{eqn:convex_Eulerian_evolution}) and thanks to the fact that $\bar{H}$ is uniformly continuous w.r.t. $x$, we have 
 $$ \psi_\varepsilon^{M,R_2}(x)  = \int_{\R^d}\int_{\R^+}\psi^{M}_\varepsilon(y)\rho(d\tau,dy)$$
    $$= \int_{\R^d}\psi^M_\varepsilon(y) Q_{x,R_1}(dy)-\frac{1}{2}\int_{\R^+}\int_{\R^d} \nabla \psi_\varepsilon^M(y)\cdot \nabla m(t,y)\,dy\, dt$$
$$    \leq  \psi_\varepsilon^{M,R_1}(x)-\int_{\R^+}\int_{\R^d}\bar{H}\big(y,\nabla \psi_\varepsilon^M(y)\big)m(t,y)\,dy\,dt $$
  $$  \leq  \psi_\varepsilon^{M,R_1}(x)+C(x)  (R_2^2-R_1^2).
  $$
Letting $R_1 \to 0^+$ and $R_2=1$, we find a uniform lower bound for $\psi_\varepsilon$
with $|x|\leq C_1$
using (\ref{eqn:Lq_bound}) and the Poincar\'{e} inequality to bound $\psi_\varepsilon$ on balls of radius $1$. This concludes the proof of (\ref{eqn:quadratic_bound}). We now use $M-\psi_\varepsilon^M$ as a test function in (\ref{Approx viscosity solution}), so we get
$$\int_{\R^d} C(M,\varepsilon)\big(M-\psi_\varepsilon^M(x)\big)e^{-V(x)}dx$$
$$\geq \int_{\R^d} \big(M-\psi_\varepsilon^M(x)\big)\Big[\frac{1}{2}\Delta \psi_\varepsilon^M(x)+\overline{H}\big(x,\nabla \psi_\varepsilon^M(x)\big)\Big]e^{-V(x)}dx$$
$$    = \int_{\R^d} \Big[\frac{1}{2}|\nabla \psi_\varepsilon^M(x)|^2 +\big(M-\psi_\varepsilon^M(x)\big)\Big(\frac{1}{2}\nabla \psi_\varepsilon^M(x)\cdot \nabla V(x)+\overline{H}\big(x,\nabla {\psi_\varepsilon^M}(x)\big) \Big)\Big]e^{-V(x)}dx$$
$$    \geq \int_{\R^d} \Big[\frac{1}{2}|\nabla \psi_\varepsilon^M(x)|^2-\big(M-\psi_\varepsilon^M(x)\big)\overline{L}\big(x,  -\frac{1}{2}\nabla V(x)\big)\Big]e^{-V(x)}dx.$$
Using the uniform quadratic lower bound on $\psi_\varepsilon^M$, the bounds on $L$ and the estimate \eqref{eqn:Lq_bound},
we get
$$
    \|\nabla \psi_\varepsilon^M\|_{L^2_V(\R^d)}^2\leq M\big(C+C(M,\varepsilon)\big) + C,
$$
which yields the estimate \eqref{eqn:H1_bound}.
To see that $\psi\wedge M$ is a distributional solution to (\ref{eqn:psi_supersolution}), we fix a smooth compactly supported test function $h\geq0$ and consider the limit
\begin{align*}
   0\geq&\ \lim_{\epsilon \rightarrow 0}\int_{\R^d}\Big[-\frac{1}{2}\nabla\psi_\varepsilon^M(x)\cdot \nabla h(x)+\overline{H}\big(x,\nabla \psi_\varepsilon^M(x)\big) h(x)\Big]dx\\
   \geq&\ \int_{\R^d}\Big[-\frac{1}{2}\nabla\psi^M(x)\cdot \nabla h(x) +\overline{H}\big(x,\nabla \psi^M(x)\big)h(x)\Big]dx,
\end{align*}
where the second inequality follows from convexity of $\overline{H}$ and hence weak lower semi-continuity of the integral.


For the last statement: assume that $\psi\wedge M$ solves {\normalfont (\ref{eqn:psi_supersolution})} weakly for all $M\geq0$.
Recalling the previous estimates, we have  
\begin{align*}
    \psi^{M,R_2}(x) - C(x) R_2^2
    \leq &\ \psi^{M,R_1}(x)-C(x) R_1^2.
\end{align*}
It follows that $R\mapsto \psi^{M,R}(x)-C(x) R^2$ is monotonically decreasing, which implies that $\psi^M$ (and so, $\psi$) is lower semi-continuous (see \cite{Sylvestre2015}).
$\qedhere$
    \end{proof}

\begin{remark}
In the case that $\psi\in H^1_V(\R^d)$, bounded above,  and $\psi$ satisfies (\ref{eqn:psi_supersolution}) in a weak sense, $\psi $ is lower semi-continuous by Proposition \ref{prop:weak_solution} and $\psi$ will satisfy (\ref{eqn:psi_supersolution}) in the viscosity sense.  Arguments for this can be found in \cite{lions1983optimal} and \cite{ishii1995equivalence}. To sketch an argument for this, we note that $\psi^R(x):=\int_{{\R^d}} \psi(y)Q_{x,R}(dy)$\, satisfies
$$
\psi^R(x) \geq \mathbb{E}^{\alpha\times \mathbb{P}}\Big[\psi(X_\tau)-\int_0^\tau \bar{L}(X_t,\beta_t)dt\Big]
$$
for all $(\mathbb{P},\beta,\alpha)\in \mathcal{A}(Q_{x,R})$.  Taking the limit as $R\rightarrow 0$ we have $\psi(x)=\lim_{R\rightarrow 0}\psi^R(x)$, and any $(\mathbb{P},\beta,\alpha)\in \mathcal{A}(\delta_x)$ can be translated to $(\mathbb{P}^R,\beta^R,\alpha^R)\in \mathcal{A}(Q_{x,R})$, showing that
$$
\psi(x) \geq \lim_{R\rightarrow 0}\mathbb{E}^{\alpha^R\times \mathbb{P}^R}\Big[\psi(X_\tau)-\int_0^\tau \bar{L}(X_t,\beta_t)dt\Big]\geq \mathbb{E}^{\alpha\times \mathbb{P}}\Big[\psi(X_\tau)-\int_0^\tau \bar{L}(X_t,\beta_t)dt\Big],
$$
which implies that $\psi$ solves (\ref{eqn:psi_supersolution}) in the sense of viscosity.
\end{remark}

\begin{lemma}\label{lem:JinB}
   The map $(t,x)\mapsto J_\psi(t,x)$ is lower semi-continuous for all $\psi\in \mathcal{B}$. 
   Moreover, under the assumptions of Proposition \ref{prop:weak_solution}, 
   we have, for any $\mu \in L^2_{-V}(\R^d)$, $\psi\mapsto \int_{\R^d}J_\psi(0,x)\mu(dx)$ is lower semi-continuous with respect to the strong $L^1_V-$ convergence
    on $\mathcal{B} \cap \{\psi  : ||\psi \wedge M||_{H^1_V} \leq C(M), \,\mbox{for any}\,\,M \geq 0
    \}$.
\end{lemma}
\begin{proof}
    The proof that $(t,x)\mapsto J_\psi(t,x)$ is lower semi-continuous is the same as for Lemma \ref{lem:measurability}. 
    For $\mu \in L_{-V}^2(\R^d)$, Lemma \ref{lem:measurability} and Theorem \ref{thm:Eulerian_duality} imply that
    \begin{align*}
        \int_{\R^d} J_\psi(0,x)\mu(dx)=\sup_{(\eta,\rho)\in \tilde{\mathcal{E}}(\mu)}\Big\{\int_{\R^d}\int_{\R^+} \psi(x)\rho(d\tau,dx)-\int_{\R^+}\int_{\R^d}\int_{\R^d}L(t,x,u)\eta_t(dx,du)dt\Big\}.
    \end{align*}
    By Lemma \ref{lem:truncation}, using the truncation $\psi^M=\psi\wedge M$, this becomes
    \begin{align*}
        =&\ \sup_{(m^{T,R},v^{T,R},\rho^{T,R})\in \mathcal{E}(\mu),\,M \geq 0}\Big\{\int_{\R^d}\int_{\R^+} \psi^M(x)\rho^{T,R}(d\tau,dx)-\int_{\R^+}\int_{\R^d}L\big(t,x,v^{T,R}\big) m^{T,R}\,dx dt\Big\},
    \end{align*}
    since we have that 
    $$  
        \liminf_{T,R,M\rightarrow \infty}\int_{\R^d}\int_{\R^+} \psi^M(x)\rho^{T,R}(d\tau,dx)\geq \int_{\R^d}\int_{\R^+} \psi(x)\rho(d\tau,dx)
    $$ 
    by lower semi-continuity of $\psi$, the quadratic lower bound (\ref{eqn:quadratic_bound}) on $\psi$ and the quadratic moment bound for $\rho$ (see Proposition \ref{prop:moment_bound}). 
    We will show that for each $(m^{T,R},v^{T,R},\rho^{T,R})\in \mathcal{E}(\mu)$ with finite cost and $M\geq 0$,
    $$
        \psi \mapsto \int_{\R^d}\int_{\R^+}\psi^M(x)\rho^{T,R}(d\tau,dx)
    $$
    is lower semi-continuous, which follows from demonstrating that the map (without the truncation) is a continuous linear functional of $H^1_V(\R^d)$.
    We now take $(m^{T,R},v^{T,R},\rho^{T,R})\in \mathcal{E}(\mu)$.
Then, for every $f\in H^1_V(\R^d)$, we have
       $$\int_{\R^d}\int_{\R^+}f(x)\rho^{T,R}(d\tau,dx)$$ \begin{align*}
      = &\ \int_{\R^d} f(x)\mu(dx)+\int_{0}^T\int_{B(0,R)}\nabla f(x)\cdot \Big(v(t,x)m^{T,R}(t,x)-\frac{1}{2}\nabla m^{T,R}(t,x)
      \Big)
    \end{align*}
    and it follows from the estimates on $m^{T,R}$ (see Theorem \ref{thm:Eulerian_regularity}) that this is a continuous linear functional, completing the proof. $\qedhere$
\end{proof}

Now, we are ready to prove attainment in the dual problem \eqref{DualStochastic}:
\begin{theorem}\label{thm:dual_attainment_large_p}
    We suppose that the Lagrangian $L$ is bounded, 
    $\mu\in L_{-V}^2(\R^d)$, $e^V\nu\in L^\infty(\R^d)$. Then, the dual problem is attained at $\psi^*\in \mathcal{B}$. Furthermore, the set $\mathcal{A}(\mu,\nu)$ is non-empty and the minimizer $(\mathbb{P},\beta,\alpha)$ of the stochastic primal problem $\mathcal{P}_L(\mu,\nu)$ satisfies
    \begin{align}\label{eqn:duality_equality}
        \mathbb{E}^{\alpha\ltimes \mathbb{P}}\Big[\psi^*(X_\tau)-\int_0^\tau L(t,X_t,\beta_t)dt\Big]=\int_{\R^d} J_{\psi^*}(0,x)\mu(dx).
    \end{align}
\end{theorem}
\begin{proof}
Let $(\psi_k)_k \subset C_b(\mathbb{R}^d)$ be a maximizing sequence in the dual problem \eqref{DualStochastic} with $(\psi_k)_V=0$, for all $k$. By Proposition \ref{prop:dual_normalization},    we can assume that, for each $k$, $\psi_k$ is lower semi-continuous and satisfies \eqref{eqn:psi_supersolution} in the sense of viscosity. Thanks to Proposition \ref{prop:weak_solution}, we also have 
$$||\psi_k||_{W^{1,1}_V(\R^d)} \leq \overline{C},\ \psi_k \geq -\overline{C}(1+|x|^2),\ \mbox{and}\ ||\psi_k^M||_{H^{1}_V(\R^d)} \leq C(M).$$\\
In particular, this implies that
    $$        \mathcal{D}_L(\mu,\nu) \leq \sup_{\psi\in \mathcal{B}} \Big\{\int_{\R^d}\psi(y)\nu(dy)-\int_{\R^d} J_\psi(0,x)\mu(dx)\Big\}\leq \mathcal{P}_L(\mu,\nu),
    $$
    which means that
    $$\mathcal{D}_L(\mu,\nu) = \sup_{\psi\in \mathcal{B}} \Big\{\int_{\R^d}\psi(y)\nu(dy)-\int_{\R^d} J_\psi(0,x)\mu(dx)\Big\}.$$
    Moreover,
    there is a function $\psi^\star \in L^1_V(\R^d)$ such that, up to a subsequence, $\psi_k \rightarrow \psi^\star$ in $L^1_V$. As $e^V \nu \in L^\infty(\R^d)$, then we have
    $$ \int_{\mathbb{R}^d} \psi_k(x)\nu(dx) \rightarrow \int_{\mathbb{R}^d} \psi^\star(x) \nu(dx).$$
    And, it is easy to see that for each $M\geq 0$, $\psi^*\wedge M$ is in $H^1_V(\R^d)$ and is a weak supersolution of (\ref{eqn:psi_supersolution}).  We then have that $\psi^*\in \mathcal{B}$ (thanks again to Proposition \ref{prop:weak_solution}). On the other hand, by Lemma \ref{lem:JinB}, we have that
$$\liminf_k \int_{\R^d} J_{\psi_k}(0,x)\,\mu(dx) \geq \int_{\R^d} J_{\psi^\star}(0,x)\,\mu(dx).$$
    Then, the existence of a maximizer for the dual problem \eqref{DualStochastic} follows.
    Finally, notice that
    $$
  \int_{\R^d} J_\psi(0,x)\mu(dx)\geq \int_{\R^d} \left[\psi(x)\wedge  0\right]\mu(dx)\geq -\|\psi\wedge 0\|_{H^1_{V}(\R^d)}\|\mu\|_{H^{-1}_{-V}(\R^d)}.
    $$
  Then, $\mathcal{D}_L(\mu,\nu)<+\infty$. This implies that $\mathcal{P}_L(\mu,\nu)$ is finite by Theorem \ref{thm:stochastic_duality} and $\mathcal{A}(\mu,\nu)$ is non-empty. The equation (\ref{eqn:duality_equality}) follows directly from the duality $\mathcal{D}_L(\mu,\nu)=\mathcal{P}_L(\mu,\nu)$ and the existence of optimizers for both problems \eqref{primal} \& \eqref{DualStochastic}. $\qedhere$ 
\end{proof}

\section{Hitting Times and Strong Solutions}\label{sec:hitting_times}
 
In this section we address the additional structure of monotonicity of $t\mapsto L(t,x,u)$. In this case the set $$R=\big\{(t,x);\ J_\psi(t,x)=\psi(x)\big\}$$ has the structure of a barrier. In particular, if \,$t\mapsto L(t,x,u)$ is increasing then for $(t,x)\in R$ and $s>t$ we have $(s,x)\in R$, and if $t\mapsto L(t,x,u)$ is decreasing then for $(t,x)\in R$ and $s<t$ we have $(s,x)\in R$.  In either of these cases strict monotonicity implies uniqueness of the optimizer. The monotonicity of $R$ in $t$ follows the same argument in \cite{GKP}.
\begin{proposition}\label{prop:J_monotonicity}
    If $t\mapsto L(t,x,u)$ is increasing then $t\mapsto J_\psi(t,x)$ is nonincreasing, and if $t\mapsto L(t,x,u)$ is decreasing then $t\mapsto J_\psi(t,x)$ is nondecreasing.
\end{proposition}
\begin{proof}
     We suppose $t\mapsto L(t,x,u)$ is increasing and select $0\leq t<s$.  We can express the value function at time $s$ by
     \begin{align*}
         J_\psi(s,x)=&\ \sup_{(\mathbb{P},\beta,\alpha)\in \mathcal{A}_t(\delta_x)}\Big\{\mathbb{E}^{\alpha\ltimes \mathbb{P}}\Big[\psi(X_\tau)-\int_t^\tau L\big(r-t+s,X_r,\beta_r\big)dr\Big]\Big\}\\
         \leq&\ \sup_{(\mathbb{P},\beta,\alpha)\in \mathcal{A}_t(\delta_x)}\Big\{\mathbb{E}^{\alpha\ltimes \mathbb{P}}\Big[\psi(X_\tau)-\int_t^\tau L\big(r,X_r,\beta_r\big)dr\Big]\Big\}=J_\psi(t,x).
     \end{align*}
     The proof in the case that $t\mapsto L(t,x,u)$ is decreasing is the same with the inequality reversed.
\end{proof}

We require a verification type theorem that will allows us to characterize the optimal process and stopping time by the dual optimizer.
   
    \begin{theorem}\label{thm:verification}
       Suppose the dual problem \,{$\mathcal{D}_L(\mu, \nu)$} is attained at $(\psi,J_\psi)$ and that $(\mathbb{P},\beta,\alpha)\in \mathcal{A}(\mu,\nu)$ minimizes the primal problem {$\mathcal{P}_L(\mu, \nu)$}.  Then,
        \begin{align}\label{eqn:stopping_critera}
            J_\psi(\tau,X_\tau)=\psi(X_\tau)
      \,\,\, \,\,\,\,\alpha\ltimes\mathbb{P}\,\,\,\,\mbox{almost surely}, \end{align}
     and
        \begin{align}\label{eqn:martingale}
            M_t:=J_\psi(t,X_t)-\int_0^t L(s,X_s,\beta_s)ds
        \end{align}
        satisfies for any $t>s$,
        \begin{align}\label{eqn:martingale_criteria}
            \mathbb{E}^{\alpha\ltimes \mathbb{P}}\Big[M_{t\wedge \tau}\Big|\tilde{\mathcal{F}}_s\Big]=M_{s\wedge \tau}
      \,\,\, \,\,\,\,\alpha\ltimes\mathbb{P}\,\,\,\mbox{almost surely}. \end{align}
    \end{theorem}
    \begin{proof}
        We first note that $J_\psi(t,x)\geq \psi(x)$ for all $(t,x)$ and $M_t$ is a supermartingale.  Then by the duality of Theorem~\ref{thm:stochastic_duality} we have
        \begin{align}
            0=&\ \mathbb{E}^{\alpha\ltimes\mathbb{P}}\Big[\psi(X_\tau)-J_\psi(0,X_0)-\int_0^\tau L(t,X_t,\beta_t)dt\Big]\\
            =&\ \mathbb{E}^{\alpha\ltimes\mathbb{P}}\Big[\psi(X_\tau)-J_\psi(\tau,X_\tau)+M_\tau-M_0\Big],
        \end{align}
       and (\ref{eqn:stopping_critera}), (\ref{eqn:martingale_criteria}) follow. $\qedhere$
    \end{proof}

We can now state a theorem that consolidates our results to show a structure to optimizers under these monotonicity { conditions on $L$}.
\begin{theorem}\label{thm:structure}
    We suppose duality and dual attainment (following from either Theorem \ref{thm:dual_attainment_large_p}, \ref{thm:dual_attainment_small_p}, or \ref{thm:dual_attaiment_1D}), and that $u\mapsto L(t,x,u)$ is strictly convex.  In addition, assume $t\mapsto L(t,x,u)$ is strictly increasing and $\psi$, $J_\psi$ are continuous, then the optimizer $(\mathbb{P},\beta,\alpha)\in \mathcal{A}(\mu,\nu)$ is unique and the optimal stopping time is given by
    \begin{align}\label{eqn:hitting_time}
        \tau^* = \inf\{t;\ J_\psi(t,X_t)=\psi(X_t)\}.
    \end{align}
    Alternatively, if \,$t\mapsto L(t,x,u)$ is strictly decreasing and $\mu$ and $\nu$ have disjoint support, then the optimizer is unique and also give by (\ref{eqn:hitting_time}).
\end{theorem}

\begin{proof}
    We first show that the optimal $(\mathbb{P},\beta,\alpha)$ for $\mathcal{P}_L(\mu,\nu)$ satisfies:
    \begin{enumerate}[label=(\roman*)]
        \item the process stops in $R$,
        $$
            \mathbb{E}^{\alpha\ltimes\mathbb{P}}\Big[ \mathbf{1}\big\{(\tau,X_\tau)\in R\big\}\Big]=1,
        $$
        \item and
        $$         \mathbb{E}^{\alpha\ltimes\mathbb{P}}\Big[\int_0^\tau \mathbf{1}\big\{(t,X_t)\in R\big\}dt\Big]=0.
        $$
    \end{enumerate}
   We then show that there is a unique $(\mathbb{P},\beta,\alpha)$ that minimizes
    $
        \mathcal{J}_L(\mathbb{P},\beta,\alpha)
    $
    subject to (i) and (ii), and the unique optimal randomized stopping time is given by $\tau^*$.
   
    By Theorem \ref{thm:verification}, we immediately have (i) from (\ref{eqn:stopping_critera}). We let $(\mathbb{P}^{t,x},\beta^{t,x},\alpha^{t,x})\in \mathcal{A}_t(\delta_x)$ be the conditional expectation, i.e.\ satisfying for each $t$ and $G\in C_b(\R^+\times \R^d\times \R^d)$,
    $$
        \mathbb{E}^{\alpha\ltimes \mathbb{P}}\Big[\int_t^\tau G(s,X_s,\beta_s)\Big]=\int_{\R^d}\mathbb{E}^{\alpha^{t,x}\ltimes \mathbb{P}^{t,x}}\Big[\int_t^\tau G(s,X_s,\beta^{t,x}_s\big)\Big]\eta_t(dx,\R^d).
    $$
    Then,
    \begin{align*}
        0=&\ \mathbb{E}^{\alpha\ltimes \mathbb{P}}\Big[M_\tau-M_{t\wedge \tau}\Big]\\
        =&\ \int_{\R^d} \mathbb{E}^{\alpha^{t,x}\ltimes \mathbb{P}^{t,x}}\Big[\psi(X_\tau)-J_\psi(t,x)+\int_t^\tau L\big(s,X_s,\beta^{t,x}_s\big)ds\Big]\eta_t(dx,\R^d),
    \end{align*}\\
    and it follows that
    $$
        J_\psi(t,x)=\mathbb{E}^{\alpha^{t,x}\ltimes \mathbb{P}^{t,x}}\Big[\psi(X_\tau)+\int_t^\tau L\big(s,X_s,\beta^{t,x}_s\big)ds\Big]
    $$
    for all $t$ and $\eta_t$-a.e.\ $x$. If $J_\psi(t,x)=\psi(x)$ and either $t\mapsto L(t,x,u)$ is strictly increasing and $s>t+\epsilon$ for $\epsilon>0$ or $t\mapsto L(t,x,u)$ is strictly decreasing and $0\leq s<t+\epsilon$ for $\epsilon<0$, then $T^{t+\epsilon-s,0}_\#\alpha^{s,x}\ltimes \mathbb{P}^{s,x}$ will satisfy
    $$
        J_\psi(t+\epsilon,x)\geq\mathbb{E}^{\alpha^{s,x}\ltimes \mathbb{P}^{s,x}}\Big[\psi(X_\tau)+\int_{s}^\tau L\big(r-s+t+\epsilon,X_r,\beta^{s,x}_{r-s+t+\epsilon}\big)dr\Big],
    $$
    which is only satisfied if $\tau=s$ with probability one, and (ii) follows.
        We now consider  $(\mathbb{P}',\beta',\alpha')\in \mathcal{A}(\mu)$ that satisfies (i) and (ii), and we have
        \begin{align*}
             \mathbb{E}^{\alpha'\ltimes\mathbb{P}'}\Big[\int_0^\tau L(t,X_t,\beta'_t)dt\Big]
            \geq \mathbb{E}^{\alpha'\ltimes\mathbb{P}'}\Big[\psi(X_{\tau})-J_\psi(0,X_0)\Big]
            = \mathbb{E}^{\alpha\ltimes\mathbb{P}}\Big[\int_0^\tau L(t,X_t,\beta_t)dt\Big],
        \end{align*}
        by the definition of $J_\psi$, optimality of $(\mathbb{P},\beta,\alpha)$ and Theorem \ref{thm:stochastic_duality},
        which implies that $(\mathbb{P}',\beta',\alpha')$ optimizes the cost over policies in $\mathcal{A}(\mu)$ that satisfy (i) and (ii).

        Given $\mathbb{P}$ and $\beta$, 
         let $\alpha^*$ be the  randomized stopping time corresponding to $\tau^*$ defined in  (\ref{eqn:hitting_time}).
       We then have that (i) and (ii) are satisfied by  $(\mathbb{P},\beta,\alpha^*)$ using continuity of $\psi$ and $J_\psi$.  Furthermore the cost of $(\mathbb{P},\beta,\alpha^*)$ is less than or equal to the cost of $(\mathbb{P},\beta,\alpha)$ because $L$ is nonnegative.  Furthermore, since the cost is equal, this implies that $\alpha^*=\alpha$. Finally, uniqueness of $\mathbb{P}$ and $\beta$ follow from strict convexity of $L$.
\end{proof}


\begin{remark}
The assumption that $\psi$ and $J_\psi$ are continuous in Theorem \ref{thm:structure} is satisfied if $1<p<2$ as in Theorem \ref{thm:dual_attainment_small_p}, and is probably not needed if the drift is bounded ($p=+\infty$) as in Theorem \ref{thm:dual_attainment_large_p} if one pursues further the Sobolev regularity of the stopping distribution $\rho^*$ for $\alpha^*$ as was done in \cite{GKPS,GKP3}.

From the control theory point of view, it is a natural question whether the optimal control policy $\beta$ satisfies the Pontryagin maximum principle: namely, $\beta_t=-D_z H(t,X_t,\nabla J_\psi(t,X_t))$ and solves the SDE: $    dX_t=\beta(X_t)dt+dW_t$  in a strong sense.  This seems to require $J_\psi$ to be $C^{1,1}$.
We leave it as an open question.
\end{remark}
\bibliography{StochasticBib}

\begin{thebibliography}{10}

\bibitem{bayraktar2019distribution}
Erhan Bayraktar and Christopher~W Miller.
\newblock Distribution-constrained optimal stopping.
\newblock {\em Mathematical Finance}, 29(1):368--406, 2019.

\bibitem{beiglbock2018geometry}
Mathias Beiglb{\"o}ck, Manu Eder, Christiane Elgert, and Uwe Schmock.
\newblock Geometry of distribution-constrained optimal stopping problems.
\newblock {\em Probability theory and related fields}, 172(1-2):71--101, 2018.

\bibitem{beiglboeck2017optimal}
Mathias Beiglboeck, Alexander~MG Cox, and Martin Huesmann.
\newblock Optimal transport and \uppercase{S}korokhod embedding.
\newblock {\em Inventiones Mathematicae}, 208(2):327--400, 2017.

\bibitem{B1}
Yann Brenier.
\newblock Polar factorization and monotone rearrangement of vector-valued
  functions.
\newblock {\em Communications on pure and applied mathematics}, 44(4):375--417,
  1991.

\bibitem{Cannarsa2010Holder}
Piermarco Cannarsa and Pierre Cardaliaguet.
\newblock H\"{o}lder estimates in space-time for viscosity solutions of
  \uppercase{H}amilton-\uppercase{J}acobi equations.
\newblock {\em Communications on Pure and Applied Mathematics}, 63(5):590--629,
  2010.

\bibitem{chen2016relation}
Yongxin Chen, Tryphon~T Georgiou, and Michele Pavon.
\newblock On the relation between optimal transport and {S}chr{\"o}dinger
  bridges: A stochastic control viewpoint.
\newblock {\em Journal of Optimization Theory and Applications},
  169(2):671--691, 2016.

\bibitem{Dyda}
Bart{\l}omiej Dyda and Moritz Kassmann.
\newblock On weighted poincar{\'e} inequalities.
\newblock {\em Ann. Acad. Sci. Fenn. Math. in print, see also http://arxiv.
  org/abs/1209.3125}, 2013.

\bibitem{ekeland1999convex}
Ivar Ekeland and Roger Temam.
\newblock {\em Convex analysis and variational problems}, volume~28.
\newblock Siam, 1999.

\bibitem{Elcrat}
A.~R. Elcrat and H.~A. MacLean.
\newblock Weighted wirtinger and poincaré inequalities on unbounded domains.
\newblock {\em Indiana University Mathematics Journal}, 29:321--332, 1980.

\bibitem{G-M}
Wilfrid Gangbo and Robert~J McCann.
\newblock The geometry of optimal transportation.
\newblock {\em Acta Mathematica}, 177(2):113--161, 1996.

\bibitem{GKP}
Nassif Ghoussoub, Young-Heon Kim, and Aaron~Zeff Palmer.
\newblock Optimal transport with controlled dynamics and free end times.
\newblock {\em SIAM Journal on Control and Optimization}, 56(5):3239--3259,
  2019.

\bibitem{GKP3}
Nassif Ghoussoub, Young-Heon Kim, and Aaron~Zeff Palmer.
\newblock A solution to the \uppercase{M}onge transport problem for
  \uppercase{B}rownian martingales.
\newblock {\em Submitted}, pages 1--23, 2019.

\bibitem{GKPS}
Nassif Ghoussoub, Young-Heon Kim, and Aaron~Zeff Palmer.
\newblock \uppercase{PDE} methods for optimal \uppercase{S}korokhod embeddings.
\newblock {\em Calc. of Variations and PDEs}, pages 1--33, 2019.

\bibitem{haussmann1986existence}
UG~Haussmann.
\newblock Existence of optimal {M}arkovian controls for degenerate diffusions.
\newblock In {\em Stochastic differential systems}, pages 171--186. Springer,
  1986.

\bibitem{ishii1995equivalence}
Hitoshi Ishii.
\newblock On the equivalence of two notions of weak solutions, viscosity
  solutions and distribution solutions.
\newblock {\em Funkcial. Ekvac}, 38(1):101--120, 1995.

\bibitem{Jacod}
Jean Jacod and Albert Shiryaev.
\newblock {\em Limit theorems for stochastic processes}, volume 288.
\newblock Springer Science \& Business Media, 2013.

\bibitem{kallblad2017dynamic}
Sigrid K{\"a}llblad.
\newblock A dynamic programming principle for distribution-constrained optimal
  stopping.
\newblock {\em arXiv preprint arXiv:1703.08534}, 2017.

\bibitem{kantorovich1942translocation}
Leonid~Vitalievich Kantorovich.
\newblock On the translocation of masses.
\newblock In {\em Dokl. Akad. Nauk. USSR (NS)}, volume~37, pages 199--201,
  1942.

\bibitem{krylov1980controlled}
Nikolaj~Vladimirovi{\v{c}} Krylov.
\newblock {\em Controlled diffusion processes}.
\newblock Springer Verlag, 1980.

\bibitem{leonard2012schrodinger}
Christian L{\'e}onard.
\newblock From the {S}chr{\"o}dinger problem to the {M}onge--{K}antorovich
  problem.
\newblock {\em Journal of Functional Analysis}, 262(4):1879--1920, 2012.

\bibitem{lions1983optimal}
Pierre-Louis Lions.
\newblock Optimal control of diffusion processes and hamilton--jacobi--bellman
  equations part 2: viscosity solutions and uniqueness.
\newblock {\em Communications in partial differential equations},
  8(11):1229--1276, 1983.

\bibitem{meyer1984tightness}
Paul~Andr{\'e} Meyer and WA~Zheng.
\newblock Tightness criteria for laws of semimartingales.
\newblock In {\em Annales de l'IHP Probabilit{\'e}s et statistiques},
  volume~20, pages 353--372, 1984.

\bibitem{mikami2002optimal}
Toshio Mikami.
\newblock Optimal control for absolutely continuous stochastic processes and
  the mass transportation problem.
\newblock {\em Electronic Communications in Probability}, 7:199--213, 2002.

\bibitem{mikami2004monge}
Toshio Mikami.
\newblock Monge's problem with a quadratic cost by the zero-noise limit of
  h-path processes.
\newblock {\em Probability theory and related fields}, 129(2):245--260, 2004.

\bibitem{mikami2008optimal}
Toshio Mikami and Michele Thieullen.
\newblock Optimal transportation problem by stochastic optimal control.
\newblock {\em SIAM Journal on Control and Optimization}, 47(3):1127--1139,
  2008.

\bibitem{Monge}
Gaspard Monge.
\newblock M{\'e}moire sur la th{\'e}orie des d{\'e}blais et des remblais.
\newblock {\em Histoire de l'Acad{\'e}mie Royale des Sciences de Paris}, 1781.

\bibitem{porretta2013planning}
Alessio Porretta.
\newblock On the planning problem for a class of mean field games.
\newblock {\em Comptes Rendus Mathematique}, 351(11-12):457--462, 2013.

\bibitem{porretta2014planning}
Alessio Porretta.
\newblock On the planning problem for the mean field games system.
\newblock {\em Dynamic Games and Applications}, 4(2):231--256, 2014.

\bibitem{Sylvestre2015}
Luis Sylvestre.
\newblock Viscosity solutions of elliptic equations.
\newblock {\em Lecture notes in Second Chicago Summer School In Analysis,
  http://math.uchicago.edu/~luis/preprints/viscosity-solutions.pdf}, 2015.

\bibitem{tan2013optimal}
Xiaolu Tan, Nizar Touzi, et~al.
\newblock Optimal transportation under controlled stochastic dynamics.
\newblock {\em The annals of probability}, 41(5):3201--3240, 2013.

\bibitem{Wong}
E~Wong.
\newblock Representation of martingales, quadratic variation and applications.
\newblock {\em SIAM J. Control}, 9:621--633, 1971.

\bibitem{zheng1985tightness}
WA~Zheng.
\newblock Tightness results for laws of diffusion processes application to
  stochastic mechanics.
\newblock In {\em Annales de l'IHP Probabilit{\'e}s et statistiques},
  volume~21, pages 103--124, 1985.

\end{thebibliography}
\bibliographystyle{plain}

\end{document}